\documentclass{article}
\usepackage{ulem}

\usepackage{amssymb,amsmath,graphicx,amsthm,mathrsfs,bm}
\usepackage[textsize=footnotesize,color=DarkGreen!40]{todonotes}

\usepackage[colorlinks=true]{hyperref}
\usepackage{pdfsync,color}
\numberwithin{equation}{section}
\newtheorem{theorem}{Theorem}[section]

\newtheorem{Lemma}[theorem]{Lemma}

\theoremstyle{definition}

\newtheorem{corollary}[theorem]{Corollary}
\theoremstyle{definition}
\newtheorem{remark}[theorem]{Remark}
\numberwithin{equation}{section}

\newcommand{\lc}
{\mathrel{\raise2pt\hbox{${\mathop<\limits_{\raise1pt\hbox
{\mbox{$\sim$}}}}$}}}

\newcommand{\gc}
{\mathrel{\raise2pt\hbox{${\mathop>\limits_{\raise1pt\hbox{\mbox{$\sim$}}}}$}}}

\newcommand{\ec}
{\mathrel{\raise2pt\hbox{${\mathop=\limits_{\raise1pt\hbox{\mbox{$\sim$}}}}$}}}

\def\bb{\begin{equation}} \def\ee{\end{equation}}

\def\beqn{\begin{eqnarray}}  \def\eqn{\end{eqnarray}}

\def\beqnx{\begin{eqnarray*}} \def\eqnx{\end{eqnarray*}}

\def\bn{\begin{enumerate}} \def\en{\end{enumerate}}

\def\bd{\begin{description}} \def\ed{\end{description}}

\def\label{\label}

\renewcommand{\leq}{\leqslant}
\renewcommand{\geq}{\geqslant}

\title{Observability inequalities for the heat equation with bounded potentials on the whole space}

\author{Yueliang Duan, \thanks{School of
Mathematics and Statistics,
Wuhan University, Wuhan 430072, China;
e-mail: duanyl@csu.edu.cn.}\quad
Lijuan Wang, \thanks{School of
Mathematics and Statistics, Computational Science Hubei Key Laboratory,
Wuhan University, Wuhan 430072, China;
e-mail: ljwang.math@whu.edu.cn.}
\quad Can Zhang
\thanks{Corresponding author. School of Mathematics and Statistics, Computational Science Hubei Key Laboratory,
Wuhan University, Wuhan 430072, China;
e-mail: canzhang@whu.edu.cn.}}

\begin{document}

\date{}

\maketitle

\begin{abstract}
In this paper we establish an observability inequality for the heat equation with bounded potentials on the whole space.  Roughly speaking, such a kind of inequality says that the total energy of solutions can be controlled by the energy localized in a subdomain,
which is  equidistributed over the whole space.
The proof of this inequality is mainly adapted from the parabolic frequency function method, which plays an important role in proving the unique continuation property for solutions of parabolic equations.  As an immediate application, we show that the null controllability holds for the heat equation with bounded potentials
on the whole space.
\end{abstract}

\medskip

\medskip


\section{Introduction and main result}
Let $N$ be a positive integer and let $T$ be a positive time. Consider the following  heat equation with a time and space dependent potential
\begin{equation}\label{1.1}
\left\{ \begin{array}{lll}
\partial_{t}\varphi-\Delta\varphi+a\varphi=0\ \ \ \ \ \ \ \ \ \ \mathrm{in}\  \mathbb{R}^{N}\times (0,T),\\
\varphi(0)=\varphi_{0} \ \ \ \ \ \ \ \ \ \ \ \ \ \ \ \ \ \ \ \ \ \ \mathrm{in}\ \mathbb{R}^{N}\\
\end{array}\right.\end{equation}
with $\varphi_{0}\in L^2(\mathbb{R}^{N})$ and $a \in L^\infty(\mathbb{R}^{N}\times(0,T)).$
According to Theorem 10.9 in \cite{Brezis1} and Theorem 4.3 in \cite{Barbuv}, (\ref{1.1}) has a unique solution
\begin{equation*}\label{1.2111}
\varphi\in L^2(0,T; H^1(\mathbb{R}^{N}))\cap C([0,T]; L^2(\mathbb{R}^{N}))\cap H^{1}(0,T;H^{-1}(\mathbb{R}^{N})).
  \end{equation*}
Moveover, for each $\delta\in(0,T),$
\begin{equation*}\label{1.3111}
\varphi\in H^{1}(\delta, T;L^2(\mathbb{R}^{N}))\cap L^2(\delta, T; H^2(\mathbb{R}^{N}))\cap C([\delta, T]; H^1(\mathbb{R}^{N})).
  \end{equation*}

Here and throughout this paper, let $r$ be a positive constant and $x_{0}\in \mathbb{R}^{N}$;
$B_{r}(x_{0})$ stands for the closed ball centered  at $x_{0}$  and of radius $r;$
$Q_{r}(x_{0})$ denotes the smallest cube centered at $x_0$ so that $B_{r}(x_{0})\subset Q_{r}(x_{0})$;
$\mathrm{int}(Q_{r}(x_{0}))$ is  the interior of $Q_{r}(x_{0});$
$\|a\|_{\infty}:=\|a\|_{L^{\infty}(\mathbb{R}^{N}\times(0,T))};$
$C(\cdot)$ denotes a generic positive constant depending on what are enclosed in the brackets.

The main result of this paper concerning the observability inequality for solutions of \eqref{1.1} is stated as follows.

 \begin{theorem}\label{Thm1}
 Let $E$ be a subset of positive measure in $(0,T)$ and let $0<r_{1}<r_{2}<+\infty$.
Assume that there is a sequence $\{x_i\}_{i\geq1}\subset\mathbb R^N$ so that
\begin{equation*}\mathbb{R}^{N}=\bigcup_{i\geq1}Q_{r_{2}}(x_{i})
\quad \text{with}\quad \mathrm{int}(Q_{r_{2}}(x_{i}))\bigcap \mathrm{int}(Q_{r_{2}}(x_{j}))=\emptyset\quad \text{for each}\quad i\neq j\in\mathbb N.
\end{equation*}
Let $$\omega\triangleq\bigcup_{i\geq1}\omega_{i}
\quad\text{with}\quad B_{r_{1}}(x_{i})\subset \omega_{i} \subset B_{r_{2}}(x_{i})\quad\text{for each}\quad i\in\mathbb N.$$
Then there exist positive constants $C=C(r_{1}, r_{2})$ and
$\widetilde{C}=\widetilde{C}(r_{1}, r_{2},E)$
so that for any $\varphi_{0}\in L^{2}(\mathbb{R}^{N})$, the corresponding solution $\varphi$ of (\ref{1.1}) satisfies 
$$
\int_{\mathbb{R}^N} |\varphi(x,T)|^2\,\mathrm dx\leq
e^{\widetilde{C}}e^{C\left(T+T\|a\|_{\infty}+\|a\|_{\infty}^{2/3}\right)}
\int_{\omega\times E}|\varphi(x,t)|^{2}\mathrm dx\mathrm dt.
$$
\end{theorem}

\medskip

Serval remarks are given below.

\begin{remark}
In the case that $E=(0,T)$,  the constant $\widetilde{C}(r_1,r_2,E)$ in the above theorem is of the form $\widetilde{C}(r_{1}, r_{2})/T$. The latter is consistent with the case of the heat equation on either bounded domains (see, e.g., \cite{AEWZ}) or the whole space
(see, e.g.,  \cite{EV17,WangZhangZhang}).
\end{remark}

\begin{remark}
It is important to point out  that in Theorem \ref{Thm1}  we obtain an observability inequality with  the same optimal dependence on the $L^\infty$-norm of the potential as in the well-known result for parabolic equations on bounded domains  (see, e.g., \cite{dzz,fz,Phung-Wang-Zhang}).
We also refer to \cite{zxz} for the similar result for Kirchhoff plate systems with potentials in unbounded domains.
\end{remark}

\begin{remark}
Quantitative unique continuation principles on multi-scale structures for Schr\"odinger and second order elliptic operators in large domains have been recently studied in \cite{DeB} and the references therein. An important feature in those works  is that the observation subdomain  satisfies a so-called equidistributed set (see Figure \ref{fig1} below for an illustration).  This indeed motivates us to impose the similar assumption on the observation subdomain $\omega\subset\mathbb R^N$ in Theorem \ref{Thm1}.
\begin{figure}[h]
\centering 
\includegraphics[width=3cm]{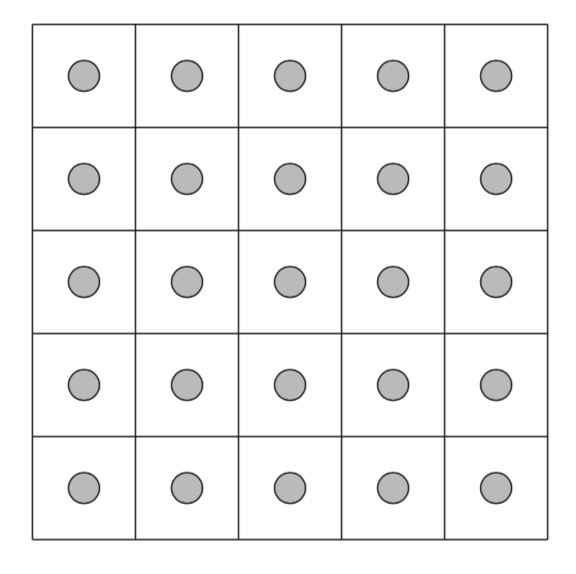}
\caption{Illustration of an 
equidistributed set in $\mathbb R^2$} 
\label{fig1}
\end{figure}
\end{remark}

\begin{remark}\label{rem1}
The result in Theorem \ref{Thm1} generalizes the observability inequality for the pure heat equation in $\mathbb R^N$ established very recently in \cite{EV17,WangZhangZhang} to that of the heat equation with space-time dependent and bounded potentials. However, the observation set $\omega$ in \cite{EV17,WangZhangZhang} is more general. To be more precise, a $\gamma$-thick set $\omega$ at scale $L$ means that in each cube $Q_L$ with the length $L$, the $N$-dimensional Lebesgue measure of $\omega\cap Q_L$ is bigger than or equals to  $\gamma L^N$. The proof therein is mainly based on quantitative estimates from measurable sets for real analytic functions.
\end{remark}

The observability inequality for parabolic equations on bounded domains has been widely
studied in past decades. When $E$ is the whole time interval  and the observation region $\omega$ is a non-empty open subset, we refer the readers to \cite{FI,fz} and a vast number of references therein for the observability inequality for parabolic equations.
In those works, the proofs are provided by the method of Carleman estimates.
When $E$ is only a subset of positive Lebesgue measure in the time interval
and the observation region $\omega$ is a non-empty open subset, we refer the readers to \cite{Phung-Wang, Phung-Wang-Zhang,WangGengsheng,wc} for the  observability inequality for parabolic equations.
More generally, when the observation subdomain is a  measurable subset of positive measure in the space and time variables,  we refer the readers to \cite{EMZ} for the observability inequality for analytic parabolic equations. The latter is mainly based on the propagation of smallness estimate for real analytic functions.

However, the studies on the observability inequality for parabolic equations on unbounded domain are rather few in last decades. We first remark that the observability inequality may not be true when the heat equation is evolving in the whole space and the observation subdomain is only a bounded and open subset (see, e.g., \cite{MZ,MZb}).
 More generally, \cite{M05a} imposed
 a condition, in terms of the Gaussian kernel, on  the  observation set so that the observability inequality for the heat equation in an unbounded domain does not hold.
Next, we would like to mention  the work \cite{CMZ} for  sufficient conditions so that the observability inequalities  hold true   for heat equations in  unbounded domains.
It showed that, for some parabolic equations in an  unbounded domain $\Omega \subset \mathbb{R}^N$, the observability inequality holds when  observations are made over a subset $\omega \subset\Omega$, with $\Omega \backslash \omega$ bounded.
For other similar results, we refer the readers to \cite{B,CMV,Gde, RM,Z16}.

Recently, there are some key progresses on this research topic.  For instance, \cite{EV17} and \cite{WangZhangZhang} independently obtained the observability inequality for the pure heat equation on the whole space, where
the observation is the thickness subset as mentioned-above in Remark \ref{rem1}.  This could be extended to
the time-independent parabolic equation associated to the Schr\"odinger operator with analytic coefficients (see \cite{EV19,leb}).
The methods utilized in these papers
are all based on the spectral inequality.
Unfortunately, they are not valid any more for the case that the coefficients
in parabolic equations are time-dependent.

The main contribution of the present paper is that we investigate a new method combined with the parabolic frequency function argument to establish the observability inequality for the heat equation with bounded and time-dependent potentials on the whole space.  More precisely, we first use the frequency function method to derive a locally quantitative estimate of unique continuation for the heat equation with a bounded potential, where we particularly quantify the dependence of the constant on the $L^\infty$-norm of the involving potential. Secondly, combined with the above local result and the geometry of the observation subdomains we
obtain a globally quantitative estimate at one time point  for solutions of the heat equation with bounded potentials. We finally utilize the so-called telescoping method to prove the desired observability inequality. It is worthing to point out that these arguments stated above are inspired from a series of works \cite{Phung-Wang-1,Phung-Wang,Phung-Wang-Zhang}.

The structure of the rest of the paper is as follows.  In Section \ref{pre}, we give several auxiliary lemmas.  They are useful in the proof of Theorem \ref{Thm1}, which will be presented in Section \ref{pro}.

\section{Preliminary lemmas}\label{pre}

First of all, we give two standard energy estimates for solutions of (\ref{1.1}).
For the sake of completeness we provide their detailed proofs in the Appendix.

\begin{Lemma}\label{lemma-1.1}
There is a constant $C_{1}>1$ so that for any $\varphi_{0}\in L^{2}(\mathbb{R}^{N})$,
the solution  $\varphi$ of (\ref{1.1}) satisfies
\begin{equation}\label{1.2}
\begin{array}{lll}
&&\displaystyle{\max_{t\in[T-\tau_{1},T]}}\int_{B_{r}(x_{0})}\varphi^{2}(x,t)\mathrm dx
+\int_{T-\tau_{1}}^{T}\int_{B_{r}(x_{0})}| \nabla\varphi(x,s)|^{2}\mathrm dx\mathrm ds\\
\\
&\leq& C_1 \left[(R-r)^{-2}+(\tau_{2}-\tau_{1})^{-1}+\|a\|_{\infty}\right]
\displaystyle{\int_{T-\tau_{2}}^{T}\int_{B_{R}(x_{0})}}\varphi^{2}(x,s)\mathrm dx\mathrm ds,\\
\\
\end{array}
\end{equation}
for all $0<r<R<+\infty$, $0<\tau_{1}<\tau_{2}<T$ and $x_{0}\in \mathbb{R}^{N}$.
\end{Lemma}

\begin{Lemma}\label{lemma-1.2}
There is a constant $C_{2}>0$ so that for any $\varphi_{0}\in L^{2}(\mathbb{R}^{N}),$
the solution $\varphi$ of (\ref{1.1}) satisfies
\begin{equation}\label{1.3}
\displaystyle{\max_{t\in[T-\tau,T]}}\int_{B_{R}(x_{0})}| \nabla\varphi(x,t)|^{2}\mathrm{d}x\leq C_{2}\big(R^{-4}+\tau^{-2}+\|a\|^{2}_{\infty}\big)\int_{T-2\tau}^{T}\int_{B_{2R}(x_{0})}\varphi^{2}
(x,s)\mathrm{d}x\mathrm{d}s,
  \end{equation}
for all $0<R<+\infty,\ 0<\tau<T/2$ and $x_{0}\in \mathbb{R}^{N}.$
\end{Lemma}

In order to give the proof of our main result, we need the following auxiliary lemma, which is motivated by \cite[Lemma 3]{Phung-Wang-Zhang}.

\begin{Lemma}\label{lemma-1.3}
Let $0<2r\leq R<+\infty$ and $\delta\in (0,1]$.
Then there are two constants $C_{3}\triangleq C_{3}(r,\delta)>0$ and
$C_{4}\triangleq C_{4}(r,\delta)>0$ so that for any  $0<\tau_{1}<\tau_{2}<T$, $x_{0}\in \mathbb{R}^{N}$,
$\varphi_{0}\in L^2(\mathbb{R}^{N})$ with $\varphi_{0}\neq0,$ the quantity
\begin{equation}\label{1.8}
h_{0}=\frac{C_{3}}{\ln\left[(1+C_{4})\left(e^{[1+2C_{1}(1+\frac{1}{r^2})]
(1+\frac{1}{\tau_{2}-\tau_{1}}+\|a\|^{2/3}_{\infty})+\frac{4C_{3}}{T}
+2T\|a\|_{\infty}}\right)\frac{\int_{T-\tau_{2}}^{T}\int_{Q_{R}(x_{0})}
\varphi^{2}(x,t)\mathrm{d}x\mathrm{d}t}{\int_{B_{r}(x_{0})}\varphi^{2}(x,T)\mathrm{d}x}\right]}
  \end{equation}
(where $C_{1}>1$ is the constant given by Lemma~\ref{lemma-1.1}),
 has the following two properties:
  \begin{description}
\item[($i$)] \begin{equation}\label{1.9}
    0<\left(1+4C_3 T^{-1}+2T\|a\|_{\infty}+\|a\|^{2/3}_{\infty}\right)h_{0}<C_{3}.
    \end{equation}
\item[($ii$)] There is a constant $C_{5}\triangleq C_{5}(r,\delta)>C_{3}$ so that
\begin{equation}\label{1.10}
e^{2T\|a\|_{\infty}}\int_{T-\tau_{2}}^{T}\int_{Q_{R}(x_{0})}\varphi^{2}\mathrm{d}x\mathrm{d}s\leq e^{1+\frac{C_{5}}{h_{0}}}\int_{B_{(1+\delta)r}(x_{0})}\varphi^2(x,t)\mathrm{d}x
\end{equation}
\end{description}
for each $t\in[T-\min\{\tau_{2},h_{0}\},T]$.
\end{Lemma}
\begin{proof}
For each $r'>0$, we write  $B_{r'}\triangleq B_{r'}(x_{0})$ and
$Q_{r'}\triangleq Q_{r'}(x_{0})$. Since $B_{2r}\subset Q_{R}$ and
$$
e^{2C_{1}\left(1+r^{-2}\right)\left[1+(\tau_{2}-\tau_{1})^{-1}+\|a\|^{2/3}_{\infty}\right]}
\geq C_{1} \left[r^{-2}+(\tau_{2}-\tau_{1})^{-1}+\|a\|_{\infty}\right],
$$
by (\ref{1.2}) (where $R$ is replaced by $2r$), we have 
\begin{eqnarray*}
&&e^{2C_1\left(1+r^{-2}\right)\left[1+(\tau_2-\tau_1)^{-1}+\|a\|^{2/3}_{\infty}\right]}
\displaystyle{\frac{\int_{T-\tau_2}^T\int_{Q_R}\varphi^2\mathrm dx\mathrm dt}{\int_{B_r}\varphi^2(x,T)\mathrm dx}}\\
&\geq&C_1\left[r^{-2}+(\tau_2-\tau_1)^{-1}+\|a\|_{\infty}\right]
\displaystyle{\frac{\int_{T-\tau_2}^T\int_{B_{2r}}\varphi^{2}\mathrm dx\mathrm dt}{\int_{B_r}\varphi^2(x,T)\mathrm dx}}\geq 1.
\end{eqnarray*}
Hence, (\ref{1.9}) follows immediately from (\ref{1.8}).

We now turn to the proof of (\ref{1.10}).
Let $h>0$, $\beta(x)=|x-x_{0}|^2$ and $\eta\in C_{0}^{\infty}(B_{(1+\delta)r})$
be such that
$$
0\leq\eta(\cdot)\leq 1\;\;\mbox{in}\;\;B_{(1+\delta)r}\;\;\mbox{and}\;\;
\eta(\cdot)=1\;\;\mbox{in}\;\;B_{(1+3\delta/4)r}.
$$
Multiplying the first equation of (\ref{1.1}) by $e^{-\beta/h}\eta^2\varphi$ and
integrating it over $B_{(1+\delta)r}$, we get 
\begin{equation}\label{1.99999}
\begin{array}{lll}
&&\displaystyle{} \frac{1}{2}\frac{\mathrm{d}}{\mathrm{d}t}\int_{B_{(1+\delta)r}}e^{-\beta/h}(\eta\varphi)^2\mathrm{d}x
+\int_{B_{(1+\delta)r}}\nabla\varphi\cdot\nabla(e^{-\beta/h}\eta^2\varphi)\mathrm{d}x\\
\\
&=&\displaystyle{}-\int_{B_{(1+\delta)r}}ae^{-\beta/h}(\eta\varphi)^2\mathrm{d}x.
\end{array}
\end{equation}
Since
$$
\nabla(e^{-\beta/h}\eta^{2}\varphi)
=-\frac{1}{h}e^{-\beta/h}\eta^{2}\varphi\nabla\beta+2e^{-\beta/h}\eta\varphi\nabla\eta+e^{-\beta/h}\eta^{2}\nabla\varphi,$$
by (\ref{1.99999}), we have 
\begin{eqnarray*}
&&\frac{1}{2}\frac{\mathrm{d}}{\mathrm{d}t}\int_{B_{(1+\delta)r}}e^{-\beta/h}(\eta\varphi)^{2}\mathrm{d}x
+\int_{B_{(1+\delta)r}}e^{-\beta/h}|\eta\nabla\varphi|^{2}\mathrm{d}x\\
&\leq&\int_{B_{(1+\delta)r}}e^{-\beta/(2h)}|\eta\nabla\varphi|
\left(\frac{2}{h}|x-x_{0}|e^{-\beta/(2h)}\eta|\varphi|+2|\nabla\eta|e^{-\beta/(2h)}|\varphi|\right)\mathrm{d}x\\
&&+\|a\|_{\infty}\int_{B_{(1+\delta)r}}e^{-\beta/h}(\eta\varphi)^{2}\mathrm{d}x.
\end{eqnarray*}
This, along with  Cauchy-Schwarz inequality, implies that
\begin{eqnarray*}
\frac{\mathrm{d}}{\mathrm{d}t}\int_{B_{(1+\delta)r}}e^{-\beta/h}(\eta\varphi)^{2}\mathrm{d}x
&\leq&\left[\frac{4(1+\delta)^2 r^{2}}{h^2}+2\|a\|_{\infty}\right]\int_{B_{(1+\delta)r}}e^{-\beta/h}(\eta\varphi)^{2}\mathrm{d}x\\
&&+4\int_{\{x:(1+3\delta/4)r\leq\sqrt{\beta(x)}\leq(1+\delta)r\}}|\nabla\eta|^{2}e^{-\beta/h}\varphi^{2}\mathrm{d}x,
\end{eqnarray*}
which indicates that
\begin{eqnarray*}
\frac{\mathrm{d}}{\mathrm{d}t}\int_{B_{(1+\delta)r}}e^{-\beta/h}(\eta\varphi)^{2}\mathrm{d}x
&\leq&\left[\frac{4(1+\delta)^2 r^{2}}{h^2}+2\|a\|_{\infty}\right]\int_{B_{(1+\delta)r}}e^{-\beta/h}(\eta\varphi)^{2}\mathrm{d}x\\
&&+4\|\nabla\eta\|^{2}_{\infty}e^{-\frac{(1+3\delta/4)^2 r^2}{h}}\int_{B_{(1+\delta)r}}\varphi^{2}\mathrm{d}x.
\end{eqnarray*}
Here and throughout the proof of Lemma~\ref{lemma-1.3},
$\|\nabla\eta\|_{\infty}\triangleq\|\nabla\eta\|_{L^{\infty}(B_{(1+\delta)r})}$.
From the latter it follows that
\begin{eqnarray*}
&&\frac{\mathrm{d}}{\mathrm{d}t}\left[e^{-\left(\frac{4(1+\delta)^2 r^2}{h^2}+2\|a\|_{\infty}\right)t}
\int_{B_{(1+\delta)r}}e^{-\beta/h}|\eta\varphi|^{2}\mathrm{d}x\right]\\
&\leq&4\|\nabla\eta\|^{2}_{\infty}e^{-\left(\frac{4(1+\delta)^2 r^2}{h^2}+2\|a\|_{\infty}\right)t}
e^{-\frac{(1+3\delta/4)^2 r^2}{h}}\int_{B_{(1+\delta)r}}\varphi^{2}\mathrm{d}x.
\end{eqnarray*}
Integrating the latter inequality over $(t,T)$, we get 
\begin{equation}\label{1.109999}
\begin{array}{lll}
&&\displaystyle{\int_{B_{(1+\delta)r}}}e^{-\beta/h}|\eta\varphi(x,T)|^{2}\mathrm dx\\
&\leq& e^{\left(\frac{4(1+\delta)^2 r^2}{h^2}+2\|a\|_{\infty}\right)(T-t)}
\displaystyle{\int_{B_{(1+\delta)r}}}e^{-\beta/h}|\eta\varphi(x,t)|^{2}\mathrm{d}x\\
&&+4e^{\left(\frac{4(1+\delta)^2 r^2}{h^2}+2\|a\|_{\infty}\right)(T-t)}\|\nabla\eta\|^2_{\infty}
e^{-\frac{(1+3\delta/4)^2 r^2}{h}}\displaystyle{\int_t^T\int_{B_{(1+\delta)r}}}\varphi^{2}(x,s)\mathrm{d}x\mathrm{d}s.
\end{array}
\end{equation}
 We simply write  $b_{1}\triangleq 4(1+\delta)^{2}, b_{2}\triangleq (1+3\delta/4)^{2}$
 and $b_{3}\triangleq (1+\delta/2)^{2}.$ It is clear that $1<b_{3}<b_{2}<b_{1}$.
 Recall that $t\leq T$. We now suppose  $h>0$ to be such that
 $$
 0<T-\frac{(b_{2}-b_{3})h}{b_{1}}\leq t.
 $$
 Then $b_{1}(T-t)/h^{2}\leq(b_{2}-b_{3})/h$ and (\ref{1.109999}) yields
\begin{eqnarray*}
\int_{B_{(1+\delta)r}}e^{-\beta/h}|\eta\varphi(x,T)|^{2}\mathrm{d}x
&\leq& e^{\frac{(b_{2}-b_{3})r^{2}}{h}}e^{2T\|a\|_{\infty}}
\int_{B_{(1+\delta)r}}e^{-\beta/h}|\eta\varphi(x,t)|^{2}\mathrm{d}x\\
&&+4\|\nabla\eta\|^{2}_{\infty}e^{2T\|a\|_{\infty}}e^{\frac{-b_{3}r^{2}}{h}}
\int_{t}^{T}\int_{B_{(1+\delta)r}}\varphi^{2}(x,s)\mathrm{d}x\mathrm{d}s.
\end{eqnarray*}
Since $\eta(\cdot)=1$ in $B_{r}$, the above estimate gives
\begin{equation}\label{1.15}
\begin{array}{lll}
\displaystyle{}\int_{B_{r}}|\varphi(x,T)|^{2}\mathrm{d}x&\leq& e^{\frac{(b_{2}-b_{3}+1)r^{2}}{h}}e^{2T\|a\|_{\infty}}
\displaystyle{\int_{B_{(1+\delta)r}}}e^{-\beta/h}|\eta\varphi(x,t)|^{2}\mathrm{d}x\\
&&+4\|\nabla\eta\|^{2}_{\infty}e^{2T\|a\|_{\infty}}e^{\frac{-(b_{3}-1)r^{2}}{h}}
\displaystyle{\int_{t}^{T}\int_{B_{(1+\delta)r}}}\varphi^{2}(x,s)\mathrm{d}x\mathrm{d}s,
\end{array}
\end{equation}
whenever $0<T-(b_{2}-b_{3})h/b_{1}\leq t\leq T$. Recall that $h_{0}<T$ from \eqref{1.9}. We choose $h$ as follows:
$$
h=\frac{b_{1}}{b_{2}-b_{3}}h_{0}=\frac{b_{1}C_{3}/(b_{2}-b_{3})}
{\ln\left[(1+C_{4})\left(e^{\left[1+2C_{1}(1+\frac{1}{r^2})\right]
(1+\frac{1}{\tau_{2}-\tau_{1}}+\|a\|^{2/3}_{\infty})+\frac{4C_{3}}{T}+2T\|a\|_{\infty}}\right)
\frac{\int_{T-\tau_{2}}^{T}\int_{Q_{R}}\varphi^{2}\mathrm{d}x\mathrm{d}t}
{\int_{B_{r}}\varphi^{2}(x,T)\mathrm{d}x}\right]}
$$
with $C_{3}\triangleq(b_{2}-b_{3})(b_{3}-1)r^{2}/b_{1}$ and
$C_{4}\triangleq4\|\nabla\eta\|^{2}_{\infty}$.
Then for any $t\in[T-\min\{\tau_{2},h_{0}\},T]$, we have 
\begin{equation}\label{2.211111}
\begin{array}{lll}
&&\displaystyle{}4\|\nabla\eta\|^{2}_{\infty}e^{2T\|a\|_{\infty}}
e^{-\frac{(b_{3}-1)r^{2}}{h}}\int_{t}^{T}\int_{B_{(1+\delta)r}}\varphi^{2}(x,s)\mathrm{d}x\mathrm{d}s\\
\\
&=&\displaystyle{}\frac{C_{4}e^{2T\|a\|_{\infty}}\int_{t}^{T}\int_{B_{(1+\delta)r}}\varphi^{2}(x,s)\mathrm{d}x\mathrm{d}s}
{(1+C_{4})\left(e^{\left[1+2C_{1}(1+\frac{1}{r^2})\right]
(1+\frac{1}{\tau_{2}-\tau_{1}}+\|a\|^{2/3}_{\infty})+\frac{4C_{3}}{T}+2T\|a\|_{\infty}}\right)
\frac{\int_{T-\tau_{2}}^{T}\int_{Q_{R}}\varphi^{2}(x,s)\mathrm{d}x\mathrm{d}s}{\int_{B_{r}}\varphi^{2}(x,T)\mathrm{d}x}}\\
\\
&\leq&\displaystyle{} \frac{1}{e}\int_{B_{r}}\varphi^{2}(x,T)\mathrm{d}x.
\end{array}
\end{equation}
(In the last inequality, we used the facts that $(1+\delta)r\leq2r\leq R\ \mathrm{ and} \ B_{(1+\delta)r}\subset Q_{R}.$)

Next, on one hand, by \eqref{1.15} and \eqref{2.211111}, we get
\begin{equation}\label{2.22222}
\left(1-\frac{1}{e}\right)\int_{B_{r}}\varphi^{2}(x,T)\mathrm{d}x\leq
e^{\frac{(b_{2}-b_{3}+1)(b_{2}-b_{3})r^{2}}{b_{1}h_{0}}}
e^{2T\|a\|_{\infty}}\int_{B_{(1+\delta)r}}|\varphi(x,t)|^{2}\mathrm{d}x
\end{equation}
for each $T-\min{\{\tau_{2},h_{0}\}}\leq t\leq T.$
On the other hand, by \eqref{1.8}, we see
\begin{equation*}
\frac{\int_{T-\tau_{2}}^{T}\int_{Q_{R}}\varphi^{2}(x,s)\mathrm{d}x\mathrm{d}s}
{\int_{B_{r}}\varphi^{2}(x,T)\mathrm{d}x}\leq e^{\frac{C_{3}}{h_{0}}},
\end{equation*}
which, combined with \eqref{2.22222}, indicates that
$$
\left(1-\frac{1}{e}\right)e^{-\frac{C_{3}}{h_{0}}}
\int_{T-\tau_{2}}^{T}\int_{Q_{R}}\varphi^{2}(x,s)\mathrm{d}x\mathrm{d}s\leq
e^{\frac{(b_{2}-b_{3}+1)(b_{2}-b_{3})r^{2}}{b_{1}h_{0}}}
e^{2T\|a\|_{\infty}}\int_{B_{(1+\delta)r}}|\varphi(x,t)|^{2}\mathrm{d}x
$$
for each $T-\min{\{\tau_{2},h_{0}\}}\leq t\leq T.$
Since $2T\|a\|_{\infty}h_{0}<C_{3}$ (see \eqref{1.9}), the desired estimate \eqref{1.10}
follows from the latter inequality immediately with $C_{5}\triangleq3C_{3}+(b_{2}-b_{3}+1)(b_{2}-b_{3})r^{2}/b_{1}$.
\end{proof}

\section{Proof of Theorem~\ref{Thm1}}\label{pro}

We first introduce the following monotonicity of the generalized frequency function
associated with parabolic equations.
\begin{Lemma}\label{lemma-2.1}(\cite{Escauriaza}, \cite{Phung-Wang-1}, \cite{Phung-Wang-Zhang})
Let $r>0$, $\lambda>0$, $T>0$ and $x_{0}\in \mathbb{R}^{N}$. Denote
$$
G_{\lambda}(x,t)\triangleq\frac{1}{(T-t+\lambda)^{N/2}}e^{-\frac{|x-x_{0}|^{2}}{4(T-t+\lambda)}},
\;\;t\in [0,T].
$$
For $u\in H^{1}(0,T; L^{2}(B_{r}(x_{0})))\cap L^{2}(0,T; H^{2}(B_{r}(x_{0}))\cap H^{1}_{0}(B_{r}(x_{0})))$ and $t\in (0,T],$ set
\begin{equation}\label{3.111}
N_{\lambda,r}(t)\triangleq\frac{\int_{B_{r}(x_{0})}|\nabla u(x,t)|^{2}G_{\lambda}(x,t)\mathrm{d}x}{\int_{B_{r}(x_{0})}|u(x,t)|^{2}G_{\lambda}(x,t)\mathrm{d}x} \;\;\;\; \text{whenever}\;\; \int_{B_{r}(x_{0})}|u(x,t)|^{2}\mathrm{d}x\neq0.
\end{equation}
The following two properties hold:
\begin{description}
\item[($i$)] \begin{equation*}
\begin{array}{lll}
    &&\displaystyle{}\frac{1}{2}\frac{\mathrm{d}}{\mathrm{d}t}
    \int_{B_{r}(x_{0})}|u(x,t)|^{2}G_{\lambda}(x,t)\mathrm{d}x
    +\int_{B_{r}(x_{0})}|\nabla u(x,t)|^{2}G_{\lambda}(x,t)\mathrm{d}x\\
    \\
    &=&\displaystyle{}\int_{B_{r}(x_{0})}u(x,t)(\partial_{t}-\Delta)u(x,t)G_{\lambda}(x,t)\mathrm{d}x.
\end{array}
    \end{equation*}

\item[($ii$)]
\begin{equation*}
\frac{\mathrm{d}}{\mathrm{d}t}N_{\lambda,r}(t)\leq\frac{1}{T-t+\lambda}N_{\lambda,r}(t)
+\frac{\int_{B_{r}(x_{0})}|(\partial_{t}u-\Delta u)(x,t)|^{2}G_{\lambda}(x,t)\mathrm{d}x}{\int_{B_{r}(x_{0})}|u(x,t)|^{2}G_{\lambda}(x,t)\mathrm{d}x}.
\end{equation*}

\end{description}
\end{Lemma}

We then have the following two-ball and one-cylinder inequality, which is inspired by \cite[Theorem 2]{Escauriaza}. Its proof here is adapted from  \cite[Lemma 4]{Phung-Wang-Zhang}
by using Lemma \ref{lemma-1.3} instead.
\begin{Lemma}\label{lemma-2.2}
Let $0<r<R<+\infty$ and  $\delta\in(0,1]$.
Then there are three positive constants $C_{6}\triangleq C_{6}(R,\delta), C_{7}\triangleq C_{7}(R,\delta)$ and $\gamma\triangleq\gamma(r,R,\delta)\in (0,1)$ so that for any $x_{0}\in \mathbb{R}^{N}$ and $\varphi_{0}\in L^2(\mathbb{R}^{N})$,
\begin{eqnarray*}
\int_{B_{R}(x_{0})}|\varphi(x,T)|^{2}\mathrm{d}x&\leq&
\left[C_{6}e^{[1+2C_{1}(1+\frac{1}{R^2})](1+\frac{4}{T}+\|a\|^{2/3}_{\infty})
+\frac{C_{7}}{T}+2T\|a\|_{\infty}}\int_{T/2}^{T}\int_{Q_{2R_{0}}(x_{0})}\varphi^{2}(x,t)\mathrm{d}x\mathrm{d}t\right]^{\gamma}\\
\\
&&\times\left(2\int_{B_{r}(x_0)}|\varphi(x,T)|^{2}\mathrm{d}x\right)^{1-\gamma},
\end{eqnarray*}
where $R_{0}\triangleq(1+2\delta)R$ and $C_{1}$ is the constant given by Lemma~\ref{lemma-1.1}.
\end{Lemma}
\begin{proof}
For each $r'>0$, we denote $B_{r'}\triangleq B_{r'}(x_{0})$ and
$Q_{r'}\triangleq Q_{r'}(x_{0})$.  Let $\chi\in C_{0}^{\infty}(B_{R_{0}})$ be such that
$$
0\leq\chi(\cdot)\leq 1 \ \mathrm{in} \ B_{R_{0}} \ \mathrm{and} \ \chi(\cdot)=1 \ \mathrm{in} \ B_{(1+3\delta/2)R}.$$
We set $u\triangleq\chi\varphi$. It is clear that
 \begin{equation}\label{3.222}
 \partial_{t}u-\Delta u=-au-2\nabla\chi\cdot\nabla\varphi-\varphi\Delta\chi \  \;\mathrm{in} \;\ B_{R_{0}}\times(0,T).
 \end{equation}
Furthermore, we define  $g\triangleq-2\nabla\chi\cdot\nabla\varphi-\varphi\Delta\chi.$\\

\textbf{Step 1}. Note that $g$ is supported on $\{x: (1+3\delta/2)R\leq|x-x_{0}|\leq R_{0}\}.$
Recall that $\chi(\cdot)=1$ in $B_{(1+\delta)R}$. We can easily check that
\begin{equation}\label{3.333}
\begin{array}{lll}
&&\displaystyle{}\frac{\int_{B_{R_{0}}} u(x,t)g(x,t)G_{\lambda}(x,t)\mathrm{d}x}
{\int_{B_{R_{0}}}|u(x,t)|^{2}G_{\lambda}(x,t)\mathrm{d}x}\\
\\
&=&\displaystyle{}\frac{\int_{B_{R_{0}}\setminus B_{(1+3\delta/2)R}}
\chi\varphi(-2\nabla\chi\cdot\nabla\varphi-\varphi\Delta\chi)
e^{-\frac{|x-x_{0}|^{2}}{4(T-t+\lambda)}}\mathrm{d}x}
{\int_{B_{R_{0}}}|\chi\varphi(x,t)|^{2}e^{-\frac{|x-x_{0}|^{2}}{4(T-t+\lambda)}}\mathrm{d}x}\\
\\
&\leq& \displaystyle{} e^{-\frac{\mathcal{K}_{1}}{T-t+\lambda}}\frac{\int_{B_{R_{0}}\setminus B_{(1+3\delta/2)R}}\left(2|\varphi\nabla\chi\cdot\nabla\varphi|
+|\Delta\chi|\varphi^{2}\right)\mathrm{d}x}{\int_{B_{(1+\delta)R}}\varphi^{2}(x,t)\mathrm{d}x},
\end{array}
\end{equation}
where $\mathcal{K}_{1}\triangleq[(1+3\delta/2)R]^2/4-[(1+\delta)R]^2/4.$
It follows from  \eqref{3.333} that
\begin{equation}\label{3.444}
\begin{array}{lll}
&&\displaystyle{}\frac{\int_{B_{R_{0}}} u(x,t)g(x,t)G_{\lambda}(x,t)\mathrm{d}x}
{\int_{B_{R_{0}}}|u(x,t)|^{2}G_{\lambda}(x,t)\mathrm{d}x}\\
\\
&\leq& \displaystyle{} e^{-\frac{\mathcal{K}_{1}}{T-t+\lambda}}\frac{2\|\nabla\chi\|_{\infty}
(\int_{B_{R_{0}}}\varphi^{2}(x,t)\mathrm{d}x)^{\frac{1}{2}}(\int_{B_{R_{0}}}
|\nabla\varphi(x,t)|^{2}\mathrm{d}x)^{\frac{1}{2}}+\|\Delta\chi\|_{\infty}\int_{B_{R_{0}}}\varphi^{2}(x,t)\mathrm{d}x}
{\int_{B_{(1+\delta)R}}\varphi^{2}(x,t)\mathrm{d}x}.
\end{array}
\end{equation}
Here $\|\nabla\chi\|_{\infty}\triangleq\|\nabla\chi\|_{L^{\infty}(B_{R_{0}})}$ and $\|\Delta\chi\|_{\infty}\triangleq\|\Delta\chi\|_{L^{\infty}(B_{R_{0}})}.$

On one hand, by Lemma~\ref{lemma-1.1} (where $r, R,\tau_{1} $ and $\tau_{2}$ are 
replaced by
$R_{0}, 2R_{0}, T/4$ and $T/2$, respectively), we have 
\begin{equation}\label{3.555}
\int_{B_{R_{0}}}\varphi^{2}(x,t)\mathrm{d}x
\leq \mathcal{K}_{2}(1+T^{-1}+\|a\|_{\infty})\int_{T/2}^{T}\int_{B_{2R_{0}}}\varphi^{2}(x,s)
\mathrm{d}x\mathrm{d}s \  \ \mathrm{for \ each} \ t\in [3T/4,T],
 \end{equation}
where $\mathcal{K}_{2}\triangleq\mathcal{K}_{2}(R)>0.$
By Lemma~\ref{lemma-1.2} (where $ R $ and $\tau$ are replaced by  $R_{0}$ and $T/4$, respectively),
we get
\begin{equation}\label{3.666}
\int_{B_{R_{0}}}|\nabla\varphi(x,t)|^{2}\mathrm{d}x\leq \mathcal{K}_{3}(1+T^{-2}+\|a\|^{2}_{\infty})\int_{T/2}^{T}\int_{B_{2R_{0}}}\varphi^{2}(x,s)\mathrm{d}x\mathrm{d}s
\  \ \mathrm{for \ each} \ t\in [3T/4,T],
\end{equation}
 where $\mathcal{K}_{3}\triangleq\mathcal{K}_{3}(R)>0.$
 It follows from \eqref{3.444}-\eqref{3.666} that
\begin{equation}\label{3.777}
\begin{array}{lll}
&&\displaystyle{}\frac{\int_{B_{R_{0}}} u(x,t)g(x,t)G_{\lambda}(x,t)\mathrm{d}x}
{\int_{B_{R_{0}}}|u(x,t)|^{2}G_{\lambda}(x,t)\mathrm{d}x}\\
\\
&\leq& \displaystyle{} e^{-\frac{\mathcal{K}_{1}}{T-t+\lambda}}
\frac{\mathcal{K}_{4}(1+T^{-2}+\|a\|^{3/2}_{\infty})
\int_{T/2}^{T}\int_{B_{2R_{0}}}\varphi^{2}(x,s)\mathrm{d}x\mathrm{d}s}
{\int_{B_{(1+\delta)R}}\varphi^{2}(x,t)\mathrm{d}x} \  \ \mathrm{for \ each} \ t\in [3T/4,T],
\end{array}
\end{equation}
 where $\mathcal{K}_{4}\triangleq\mathcal{K}_{4}(R, \delta)>0.$
 
 According to \eqref{1.10} (where $r, R,\tau_{1} $ and $\tau_{2}$ are replaced by
 $R, 2R_{0}, T/4$ and $T/2$, respectively), it holds that
 \begin{equation}\label{3.888}
\begin{array}{lll}
\displaystyle{}e^{2T\|a\|_{\infty}}\int_{T/2}^{T}\int_{B_{2R_{0}}}\varphi^{2}\mathrm{d}x\mathrm{d}s&\leq& \displaystyle{}e^{2T\|a\|_{\infty}}\int_{T/2}^{T}\int_{Q_{2R_{0}}}\varphi^{2}\mathrm{d}x\mathrm{d}s\\
\\
&\leq& \displaystyle{} e^{1+\frac{C_{5}}{h_{0}}}
\int_{B_{(1+\delta)R}}\varphi^{2}(x,t)\mathrm{d}x \  \ \mathrm{for \ each} \ t\in [T-h_{0},T].
\end{array}
\end{equation}
Here, we used the fact that $h_{0}<T/4$ (see \eqref{1.9}). This, along with \eqref{3.777}, implies that
\begin{equation}\label{3.999}
\begin{array}{lll}
\displaystyle{}\frac{\int_{B_{R_{0}}} u(x,t)g(x,t)G_{\lambda}(x,t)\mathrm{d}x}
{\int_{B_{R_{0}}}|u(x,t)|^{2}G_{\lambda}(x,t)\mathrm{d}x}\leq \displaystyle{} \mathcal{K}_{4}e^{-\frac{\mathcal{K}_{1}}{T-t+\lambda}}e^{1+\frac{C_{5}}{h_{0}}}(1+T^{-2})
 \ \ \  \mathrm{for \ each}  \ \ t\in [T-h_{0},T].
\end{array}
\end{equation}

On the other hand, by similar arguments as those for \eqref{3.444}, we have 
 \begin{eqnarray*}
&&\int_{t}^{T}\frac{\int_{B_{R_{0}}} |g(x,s)|^{2}G_{\lambda}(x,s)\mathrm{d}x}{\int_{B_{R_{0}}}|u(x,s)|^{2}G_{\lambda}(x,s)\mathrm{d}x}\mathrm{d}s\\
&\leq& \int_{t}^{T}\frac{\int_{B_{R_{0}}} |-2\nabla\chi\cdot\nabla\varphi-\varphi\Delta\chi|^{2}\mathrm{d}x}
{\int_{B_{(1+\delta)R}}|\varphi(x,s)|^{2}\mathrm{d}x}e^{-\frac{\mathcal{K}_{1}}{T-s+\lambda}}\mathrm{d}s\\
&\leq&\int_{t}^{T}\frac{8\|\nabla\chi\|^{2}_{\infty}\int_{B_{R_{0}}} |\nabla\varphi|^{2}\mathrm{d}x+2\|\Delta\chi\|^{2}_{\infty}
\int_{B_{R_{0}}}\varphi^{2}\mathrm{d}x}{\int_{B_{(1+\delta)R}}|\varphi(x,s)|^{2}\mathrm{d}x}
e^{-\frac{\mathcal{K}_{1}}{T-s+\lambda}}\mathrm{d}s
\end{eqnarray*}
 for each $t\in [3T/4,T].$ This, together with  \eqref{3.555} and  \eqref{3.666}, yields
\begin{equation}\label{3.101010}
\begin{array}{lll}
\displaystyle{}\int_{t}^{T}\frac{\int_{B_{R_{0}}} |g(x,s)|^{2}G_{\lambda}(x,s)\mathrm{d}x}{\int_{B_{R_{0}}}|u(x,s)|^{2}G_{\lambda}(x,s)\mathrm{d}x}\mathrm{d}s
\leq\displaystyle{}\mathcal{K}_{5}(1+T^{-2}+\|a\|^{2}_{\infty})
\int_{t}^{T}\frac{\int_{T/2}^{T}\int_{B_{2R_{0}}}\varphi^{2}\mathrm{d}x\mathrm{d}s}
{\int_{B_{(1+\delta)R}}|\varphi(x,s)|^{2}\mathrm{d}x}e^{-\frac{\mathcal{K}_{1}}{T-s+\lambda}}\mathrm{d}s
\end{array}
\end{equation}
for each  $t\in [3T/4,T],$ where $\mathcal{K}_{5}\triangleq\mathcal{K}_{5}(R,\delta)>0.$
It follows from \eqref{3.888} and \eqref{3.101010} that
\begin{equation}\label{3.111111}
\begin{array}{lll}
&&\displaystyle{}\int_{t}^{T}\frac{\int_{B_{R_{0}}} |g(x,s)|^{2}G_{\lambda}(x,s)\mathrm{d}x}{\int_{B_{R_{0}}}|u(x,s)|^{2}G_{\lambda}(x,s)\mathrm{d}x}\mathrm{d}s\\
&\leq&\displaystyle{}\mathcal{K}_{5}(1+T^{-2}+\|a\|^{2}_{\infty})
e^{1+\frac{C_{5}}{h_{0}}}e^{-2T\|a\|_{\infty}}\int_{t}^{T}e^{-\frac{\mathcal{K}_{1}}{T-s+\lambda}}\mathrm{d}s\\
&\leq&\displaystyle{}\mathcal{K}_{5}(1+T^{-2})
e^{1+\frac{C_{5}}{h_{0}}}e^{-\frac{\mathcal{K}_{1}}{T-t+\lambda}}(T-t) \ \ \mathrm{for \ each} \  \ t\in [T-h_{0},T].
\end{array}
\end{equation}

\vskip 3pt

\textbf{Step 2}. In this step, our plan is to give an estimate about
$\lambda N_{\lambda,R_{0}}(T)$ (see \eqref{3.111}).
By $(ii)$ of Lemma~\ref{lemma-2.1} and \eqref{3.222}, we get 
$$
\frac{\mathrm{d}}{\mathrm{d}t}N_{\lambda,R_{0}}(t)\leq
\frac{1}{T-t+\lambda}N_{\lambda,R_{0}}(t)+\frac{\int_{B_{R_{0}}}|(-au+g)(x,t)|^{2}G_{\lambda}(x,t)\mathrm{d}x}
{\int_{B_{R_{0}}}|u(x,t)|^{2}G_{\lambda}(x,t)\mathrm{d}x},
$$
which indicates that
 \begin{eqnarray*}
\frac{\mathrm{d}}{\mathrm{d}t}\left[(T-t+\lambda)N_{\lambda,R_{0}}(t)\right]
&\leq&(T-t+\lambda)\frac{\int_{B_{R_{0}}}|(-au+g)(x,t)|^{2}G_{\lambda}(x,t)\mathrm{d}x}
{\int_{B_{R_{0}}}|u(x,t)|^{2}G_{\lambda}(x,t)\mathrm{d}x}\\
&\leq&2(T-t+\lambda)\left(\|a\|^{2}_{\infty}
+\frac{\int_{B_{R_{0}}}|g(x,t)|^{2}G_{\lambda}(x,t)\mathrm{d}x}{\int_{B_{R_{0}}}|u(x,t)|^{2}G_{\lambda}(x,t)\mathrm{d}x}\right).
\end{eqnarray*}
From  the latter inequality it follows that
 \begin{eqnarray*}
 \lambda N_{\lambda,R_{0}}(T)&\leq&(T-t+\lambda)N_{\lambda,R_{0}}(t)+2\|a\|^{2}_{\infty}\int_{t}^{T}(T-s+\lambda)\mathrm{d}s\\
 &&+2\int_{t}^{T}(T-s+\lambda)\frac{\int_{B_{R_{0}}}|g(x,s)|^{2}G_{\lambda}(x,s)\mathrm{d}x}
 {\int_{B_{R_{0}}}|u(x,s)|^{2}G_{\lambda}(x,s)\mathrm{d}x}\mathrm{d}s.
 \end{eqnarray*}
Hence, for any $0<T-\varepsilon\leq t<T$ (where $\varepsilon$ will be determined later), we have 
\begin{equation}\label{3.121212}
\frac{\lambda}{\varepsilon+\lambda}N_{\lambda,R_{0}}(T)\leq N_{\lambda,R_{0}}(t)+2\varepsilon\|a\|^{2}_{\infty}+2\int_{t}^{T}
\frac{\int_{B_{R_{0}}}|g(x,s)|^{2}G_{\lambda}(x,s)\mathrm{d}x}{\int_{B_{R_{0}}}|u(x,s)|^{2}G_{\lambda}(x,s)\mathrm{d}x}\mathrm{d}s.
\end{equation}
Moreover, by $(i)$ of Lemma~\ref{lemma-2.1} and \eqref{3.222}, we observe that
\begin{eqnarray*}
&&\frac{1}{2}\frac{\mathrm{d}}{\mathrm{d}t}\int_{B_{R_{0}}}|u(x,t)|^{2}G_{\lambda}(x,t)
\mathrm{d}x+N_{\lambda,R_{0}}(t)\int_{B_{R_{0}}}| u(x,t)|^{2}G_{\lambda}(x,t)\mathrm{d}x\\
&=&-\int_{B_{R_{0}}}a(x,t)|u(x,t)|^{2}G_{\lambda}(x,t)\mathrm{d}x\\
&&
+\frac{\int_{B_{R_{0}}}u(x,t)g(x,t)G_{\lambda}(x,t)\mathrm{d}x}
{\int_{B_{R_{0}}}|u(x,t)|^{2}G_{\lambda}(x,t)\mathrm{d}x}\int_{B_{R_{0}}}|u(x,t)|^{2}G_{\lambda}(x,t)\mathrm{d}x.
\end{eqnarray*}
This, along with \eqref{3.121212}, implies 
\begin{equation}\label{3.131313}
\begin{array}{lll}
&&\displaystyle{}\frac{1}{2}\frac{\mathrm{d}}{\mathrm{d}t}
\int_{B_{R_{0}}}|u(x,t)|^{2}G_{\lambda}(x,t)\mathrm{d}x+\frac{\lambda}
{\varepsilon+\lambda}N_{\lambda,R_{0}}(T)\int_{B_{R_{0}}}|u(x,t)|^{2}G_{\lambda}(x,t)\mathrm{d}x\\
&\leq&\displaystyle{}(\|a\|_{\infty}+2\varepsilon\|a\|^{2}_{\infty})
\int_{B_{R_{0}}}|u(x,t)|^{2}G_{\lambda}(x,t)\mathrm{d}x+\int_{B_{R_{0}}}|u(x,t)|^{2}G_{\lambda}(x,t)\mathrm{d}x\\
&&\displaystyle{}\times\left(\frac{\int_{B_{R_{0}}}u(x,t)g(x,t)G_{\lambda}(x,t)\mathrm{d}x}
{\int_{B_{R_{0}}}|u(x,t)|^{2}G_{\lambda}(x,t)\mathrm{d}x}+2\int_{t}^{T}
\frac{\int_{B_{R_{0}}}|g(x,s)|^{2}G_{\lambda}(x,s)\mathrm{d}x}
{\int_{B_{R_{0}}}|u(x,s)|^{2}G_{\lambda}(x,s)\mathrm{d}x}\mathrm{d}s\right).
\end{array}
\end{equation}

Next, on one hand, it follows from  (\ref{3.999}) and (\ref{3.111111}) that
\begin{equation}\label{2.7}
\begin{array}{lll}
&&\displaystyle{}\frac{\int_{B_{R_{0}}}u(x,t)g(x,t)G_{\lambda}(x,t)\mathrm{d}x}
{\int_{B_{R_{0}}}|u(x,t)|^{2}G_{\lambda}(x,t)\mathrm{d}x}+2\int_{t}^{T}
\frac{\int_{B_{R_{0}}}|g(x,s)|^{2}G_{\lambda}(x,s)\mathrm{d}x}
{\int_{B_{R_{0}}}|u(x,s)|^{2}G_{\lambda}(x,s)\mathrm{d}x}\mathrm{d}s\\
&\leq& \mathcal{K}_{6}\left(1+\varepsilon\right)\left(1+T^{-2}\right)
e^{-\frac{\mathcal{K}_{1}}{\varepsilon+\lambda}}e^{\frac{C_{5}+\mathcal{K}_{1}}{h_{0}}}\\
&\triangleq&Q_{h_{0},\varepsilon,\lambda}
 \ \ \mathrm{for\ each} \ 0<T-\varepsilon\leq t<T \ \mathrm{ with} \  \varepsilon\in (0,h_{0}],
\end{array}
\end{equation}
where $\mathcal{K}_{6}\triangleq\mathcal{K}_{6}(R,\delta)>0.$
This, along with (\ref{3.131313}), implies that
\begin{eqnarray*}
&&\frac{1}{2}\frac{\mathrm{d}}{\mathrm{d}t}\int_{B_{R_{0}}}|u(x,t)|^{2}G_{\lambda}(x,t)\mathrm{d}x\\
&\leq& -\left(\frac{\lambda}{\varepsilon+\lambda}N_{\lambda,R_{0}}(T)
-\|a\|_{\infty}-2\varepsilon\|a\|^{2}_{\infty}-Q_{h_{0},\varepsilon,\lambda}\right)
\int_{B_{R_{0}}}| u(x,t)|^{2}G_{\lambda}(x,t)\mathrm{d}x,
\end{eqnarray*}
which indicates that
$$\frac{\mathrm{d}}{\mathrm{d}t}\left[e^{2\left(\frac{\lambda}
{\varepsilon+\lambda}N_{\lambda,R_{0}}(T)-
\|a\|_{\infty}-2\varepsilon\|a\|^{2}_{\infty}-Q_{h_{0},\varepsilon,\lambda}\right)t}\int_{B_{R_{0}}}| u(x,t)|^{2}G_{\lambda}(x,t)\mathrm{d}x\right]\leq 0
$$
for each $0<T-\varepsilon\leq t<T$ with $\varepsilon\in (0,h_{0}]$.
Integrating the latter inequality over $(T-\varepsilon,T-\varepsilon/2)$, we obtain 
\begin{eqnarray*}
&&e^{\frac{\varepsilon\lambda}{\varepsilon+\lambda}N_{\lambda,R_{0}}(T)}
\int_{B_{R_{0}}}|u(x,T-\varepsilon/2)|^{2}G_{\lambda}(x,T-\varepsilon/2)\mathrm{d}x\\
&\leq& e^{\varepsilon\|a\|_{\infty}+2\varepsilon^{2}\|a\|^{2}_{\infty}
+\varepsilon Q_{h_{0},\varepsilon,\lambda}}\int_{B_{R_{0}}}|u(x,T-\varepsilon)|^{2}G_{\lambda}(x,T-\varepsilon)\mathrm{d}x.
\end{eqnarray*}
This yields 
\begin{equation}\label{3.151515}
e^{\frac{\varepsilon\lambda}{\varepsilon+\lambda}N_{\lambda,R_{0}}(T)}\leq e^{\varepsilon\|a\|_{\infty}+2\varepsilon^{2}\|a\|^{2}_{\infty}+\varepsilon Q_{h_{0},\varepsilon,\lambda}}\frac{\int_{B_{R_{0}}}| u(x,T-\varepsilon)|^{2}e^{-\frac{|x-x_{0}|^{2}}{4(\varepsilon+\lambda)}}\mathrm{d}x}{\int_{B_{R_{0}}}| u(x,T-\varepsilon/2)|^{2}e^{-\frac{|x-x_{0}|^{2}}{4(\varepsilon/2+\lambda)}}\mathrm{d}x}.
\end{equation}

On the other hand, by (\ref{3.555}), we see
\begin{eqnarray*}
\frac{\int_{B_{R_{0}}}| u(x,T-\varepsilon)|^{2}
e^{-\frac{|x-x_{0}|^{2}}{4(\varepsilon+\lambda)}}\mathrm{d}x}{\int_{B_{R_{0}}}| u(x,T-\varepsilon/2)|^{2}e^{-\frac{|x-x_{0}|^{2}}{4(\varepsilon/2+\lambda)}}\mathrm{d}x}
&\leq&\frac{e^{\frac{((1+\delta)R)^{2}}{4(\varepsilon/2+\lambda)}}\int_{B_{R_{0}}}| \varphi(x,T-\varepsilon)|^{2}\mathrm{d}x}{\int_{B_{(1+\delta)R}}| \varphi(x,T-\varepsilon/2)|^{2}\mathrm{d}x}\\
&\leq&\frac{e^{\frac{((1+\delta)R)^{2}}{2\varepsilon}}\mathcal{K}_{2}(1+T^{-1}+\|a\|_{\infty})
\int_{T/2}^{T}\int_{Q_{2R_{0}}}\varphi^{2}\mathrm{d}x\mathrm{d}t}
{\int_{B_{(1+\delta)R}}| \varphi(x,T-\varepsilon/2)|^{2}\mathrm{d}x},
\end{eqnarray*}
which, combined with $(ii)$ of Lemma~\ref{lemma-1.3}
(where $r, R,\tau_{1} $ and $\tau_{2}$ are replaced by  $R, 2R_{0}, T/4$ and $T/2$, respectively),
indicates that
\begin{equation}\label{3.161616}
\begin{array}{lll}
\displaystyle{}\frac{\int_{B_{R_{0}}}| u(x,T-\varepsilon)|^{2}
e^{-\frac{|x-x_{0}|^{2}}{4(\varepsilon+\lambda)}}\mathrm{d}x}{\int_{B_{R_{0}}}| u(x,T-\varepsilon/2)|^{2}e^{-\frac{|x-x_{0}|^{2}}{4(\varepsilon/2+\lambda)}}\mathrm{d}x}&\leq& \displaystyle{}\frac{e^{\frac{((1+\delta)R)^{2}}{2\varepsilon}}\mathcal{K}_{2}(1+T^{-1}
+\|a\|_{\infty})e^{1+\frac{C_{5}}{h_{0}}}}{e^{2T\|a\|_{\infty}}}\\
\\
&\leq&\displaystyle{}\mathcal{K}_{2}e^{\frac{((1+\delta)R)^{2}}{2\varepsilon}}\left(1+T^{-1}\right)e^{1+\frac{C_{5}}{h_{0}}}.
\end{array}
\end{equation}
Then, it follows from (\ref{3.151515}) and (\ref{3.161616}) that for each  $\varepsilon\in (0,h_{0}]$,
\begin{equation}\label{3.171717}
\begin{array}{lll}
&&\lambda N_{\lambda,R_{0}}(T)\\
&\leq&
\displaystyle{\frac{\varepsilon+\lambda}{\varepsilon}}
\left[\varepsilon \|a\|_{\infty}+2\varepsilon^2\|a\|^{2}_{\infty}
+\varepsilon Q_{h_{0},\varepsilon,\lambda}+\displaystyle{\frac{(1+\delta)^2 R^2}{2\varepsilon}}
+1+\displaystyle{\frac{C_{5}}{h_{0}}}+\ln\left(\mathcal{K}_{2}\left(1+T^{-1}\right)\right)\right].
\end{array}
\end{equation}
Finally, we choose $\lambda=\mu\varepsilon$ with $\mu\in(0,1)$ (which will be determined later) and $\varepsilon=\frac{\mathcal{K}_{1}h_{0}}{2(C_{5}+\mathcal{K}_{1})}$ so that
$Q_{h_{0},\varepsilon,\lambda}$ (see (\ref{2.7})) satisfies
\begin{equation}\label{3.181818}
\begin{array}{lll}
Q_{h_{0},\varepsilon,\lambda}&=&\mathcal{K}_{6}
\left(1+\varepsilon\right)\left(1+T^{-2}\right)e^{\frac{C_{5}+\mathcal{K}_{1}}{h_{0}}(\frac{\mu-1}{\mu+1})}\\
\\
&\leq& \mathcal{K}_{6}\left(1+\varepsilon\right)\left(1+T^{-2}\right).
\end{array}
\end{equation}
Since $\varepsilon\leq h_{0}$, by  (\ref{3.171717}) and (\ref{3.181818}), we get 
\begin{equation}\label{3.191919}
\begin{array}{lll}
 \lambda N_{\lambda,R_{0}}(T)&\leq& 2\left[h_{0}\|a\|_{\infty}+2h_{0}^{2}\|a\|^{2}_{\infty}
 +\mathcal{K}_{6}\left(1+\varepsilon\right)\left(1+T^{-2}\right)h_{0} \right]\\
 \\
 &&+2\left[1+\frac{C_{5}}{h_{0}}+\mathcal{K}_{2}
 \left(1+T^{-1}\right)+\frac{C_{5}+\mathcal{K}_{1}}{\mathcal{K}_{1}h_{0}}(1+\delta)^2 R^2\right].
\end{array}
\end{equation}

According to $(i)$ of Lemma~\ref{lemma-1.3}
(where $r, R,\tau_{1} $ and $\tau_{2}$ are replaced by  $R, 2R_{0}, T/4$ and $T/2$, respectively),
it is clear that
$$h_{0}<C_{3},\ h_{0}<T,\ h_{0}T\|a\|_{\infty}<C_{3} \ \mathrm{and} \ h^{3}_{0}\|a\|^{2}_{\infty}<C^{3}_{3}.
$$
These, together with (\ref{3.191919}), derive that 
 \begin{eqnarray*}
 \varepsilon\lambda N_{\lambda,R_{0}}(T)&\leq& 2\left[h_{0}T\|a\|_{\infty}+2 h_{0}^{3}\|a\|^{2}_{\infty}+\mathcal{K}_{6}\left(1+h_{0}\right)\left(h_{0}^{2}+h_{0}^{2}T^{-2}\right) \right]\\
 &&+2\left[h_{0}+C_{5}+\mathcal{K}_{2}h_{0}\left(1+T^{-1}\right)
 +\frac{C_{5}+\mathcal{K}_{1}}{\mathcal{K}_{1}}(1+\delta)^2 R^{2}\right]\\
 &\leq& 2\left[C_{3}+2 C_{3}^{3}+\mathcal{K}_{6}\left(1+C_{3}\right)\left(1+C_{3}^{2}\right) \right]\\
 &&+2\left[C_{3}+C_{5}+\mathcal{K}_{2}\left(1+C_{3}\right)
 +\frac{C_{5}+\mathcal{K}_{1}}{\mathcal{K}_{1}}(1+\delta)^2 R^{2}\right].
 \end{eqnarray*}
 Hence,
 \begin{equation}\label{2.9}
 \frac{16\lambda}{r^{2}}\left(\frac{N}{4}+\lambda N_{\lambda,R_{0}}(T)\right)
 \leq\frac{16\mu}{r^{2}}\left(\frac{N}{4}C_{3}+\varepsilon\lambda N_{\lambda,R_{0}}(T)\right)\leq\mu(1+\mathcal{K}_{7}),
 \end{equation}
where $\mathcal{K}_{7}\triangleq\mathcal{K}_{7}(r,R,\delta)>0.$\\

 \textbf{Step 3}. We claim that
 \begin{equation}\label{3.212121}
\begin{array}{lll}
\displaystyle{} \int_{B_{R_{0}}}| u(x,T)|^{2}e^{-\frac{|x-x_{0}|^{2}}{4\lambda}}\mathrm{d}x
&\leq&\displaystyle{}\int_{B_{r}}|\varphi(x,T)|^{2}e^{-\frac{|x-x_{0}|^{2}}{4\lambda}}\mathrm{d}x\\
 \\
 &&\displaystyle{}+\mu(1+\mathcal{K}_{7})\int_{B_{R_{0}}}|u(x,T)|^{2}e^{-\frac{|x-x_{0}|^{2}}{4\lambda}}\mathrm{d}x.
\end{array}
 \end{equation}
Indeed, since $B_{R_{0}}$ is star-shaped with respect to $x_{0}$, we have 
(see, for instance, \cite{Escauriaza}, \cite{Phung-Wang-1} or \cite{Phung-Wang-Zhang})
 \begin{eqnarray*}
 &&\frac{1}{16\lambda}\int_{B_{R_{0}}}|x-x_{0}|^{2}| u(x,T)|^{2}e^{-\frac{|x-x_{0}|^{2}}{4\lambda}}\mathrm{d}x\\
 &\leq&\frac{N}{4}\int_{B_{R_{0}}}|u(x,T)|^{2}e^{-\frac{|x-x_{0}|^{2}}{4\lambda}}\mathrm{d}x+\lambda\int_{B_{R_{0}}}|\nabla u(x,T)|^{2}e^{-\frac{|x-x_{0}|^{2}}{4\lambda}}\mathrm{d}x.
 \end{eqnarray*}
This implies 
\begin{equation}\label{3.222222}
\begin{array}{lll}
&&\displaystyle{}\int_{B_{R_{0}}}|u(x,T)|^{2}e^{-\frac{|x-x_{0}|^{2}}{4\lambda}}\mathrm{d}x\\
\\
 &\leq&\displaystyle{} \int_{B_{R_{0}}\setminus B_{r}}
 \frac{|x-x_{0}|^{2}}{r^{2}}| u(x,T)|^{2}e^{-\frac{|x-x_{0}|^{2}}{4\lambda}}\mathrm{d}x
 \displaystyle{}+\int_{B_{r}}|u(x,T)|^{2}e^{-\frac{|x-x_{0}|^{2}}{4\lambda}}\mathrm{d}x\\
 \\
 &\leq&\displaystyle{}\frac{16\lambda}{r^{2}}
 \left[\frac{N}{4}+\lambda N_{\lambda,R_{0}}(T)\right]\int_{B_{R_{0}}}|u(x,T)|^{2}
 e^{-\frac{|x-x_{0}|^{2}}{4\lambda}}\mathrm{d}x
 +\displaystyle{}\int_{B_{r}}|\varphi(x,T)|^{2}e^{-\frac{|x-x_{0}|^{2}}{4\lambda}}\mathrm{d}x  ,
\end{array}
\end{equation}
 where in the last line, we used the definition of $N_{\lambda,R_{0}}(T)$
 and the fact that $u=\varphi$ in $B_{r}$.
 Then (\ref{3.212121}) follows from  (\ref{3.222222}) and  (\ref{2.9}) immediately.\\

 \textbf{Step 4}. End of the proof.

 We choose $\mu=1/[2(1+\mathcal{K}_{7})]$. Then, $\lambda=\mu\varepsilon=\frac{\mathcal{K}_{1}h_{0}}{4(1+\mathcal{K}_{7})(C_{5}+\mathcal{K}_{1})}.$
 By (\ref{3.212121}), we have 
 \begin{eqnarray*}
 \int_{B_{R}}|\varphi(x,T)|^{2}dx&\leq&
 e^{\frac{R^{2}}{4\lambda}}\int_{B_{R_{0}}}|u(x,T)|^{2}e^{-\frac{|x-x_{0}|^{2}}{4\lambda}}\mathrm{d}x\\
 &\leq&2e^{\frac{R^{2}}{4\lambda}}\int_{B_{r}}|\varphi(x,T)|^{2}
 e^{-\frac{|x-x_{0}|^{2}}{4\lambda}}\mathrm{d}x.
 \end{eqnarray*}
 This, along with the definition of $h_{0}$ (see (\ref{1.8}),
 where $r, R,\tau_{1} $ and $\tau_{2}$ are replaced by  $R, 2R_{0}, T/4$ and $T/2$, respectively), implies that
 \begin{eqnarray*}
 &&\int_{B_{R}}|\varphi(x,T)|^{2}dx
 \leq2e^{\frac{(1+\mathcal{K}_{7})(C_{5}+\mathcal{K}_{1})R^{2}}
 {\mathcal{K}_{1}h_{0}}}\int_{B_{r}}|\varphi(x,T)|^{2}\mathrm{d}x\\
 &\leq&\left[(1+C_{4})\left(e^{[1+2C_{1}(1+R^{-2})]
 (1+4T^{-1}+\|a\|^{2/3}_{\infty})+\frac{4C_{3}}{T}
 +2T\|a\|_{\infty}}\right)\frac{\int_{T/2}^{T}\int_{Q_{2R_{0}}}\varphi^{2}\mathrm{d}x\mathrm{d}t}
 {\int_{B_{R}}|\varphi(x,T)|^{2}\mathrm{d}x}\right]^{\frac{(1+\mathcal{K}_{7})(C_{5}+\mathcal{K}_{1})R^{2}}
 {\mathcal{K}_{1}C_{3}}}\\
 &&\times2\int_{B_{r}}|\varphi(x,T)|^{2}\mathrm{d}x.
 \end{eqnarray*}
Hence, we can conclude that the desired estimate of Lemma~\ref{lemma-2.2} holds  with
$$
\gamma=\frac{(1+\mathcal{K}_{7})(C_{5}+\mathcal{K}_{1})R^{2}}
{C_{3}\mathcal{K}_{1}+(1+\mathcal{K}_{7})(C_{5}+\mathcal{K}_{1})R^{2}}\in(0,1).
$$

\vskip 3pt

In summary, we finish the proof of this lemma.
\end{proof}

At last, based on Lemma \ref{lemma-2.2} we will derive an interpolation inequality
for solutions of \eqref{1.1} at one time point, which is analogous to those established
for parabolic equations in bounded domains; see \cite[Theorem 6]{AEWZ} and \cite[Lemma 5]{Phung-Wang-Zhang} for instance.

\begin{Lemma}\label{2.3} Let $0<r<R<+\infty$.
Suppose that $\mathbb{R}^{N}=\cup_{i\geq1}Q_{R}(\widetilde{x}_{i})$,
where $\mathrm{int}(Q_{R}(\widetilde{x}_{i}))\cap \mathrm{int}(Q_{R}(\widetilde{x}_{j}))=\emptyset$
for $i\neq j;$ $\widetilde{\omega}=\cup_{i\geq1}\widetilde{\omega}_{i}$,
where $\widetilde{\omega}_{i}, i\geq1$, is a nonempty open subset with
$B_{r}(\widetilde{x}_{i})\subset \widetilde{\omega}_{i} \subset B_{R}(\widetilde{x}_{i}).$
Then there are two constants $C_{8}\triangleq C_{8}(R)>0$ and
$\theta\triangleq\theta(r,R)\in (0,1)$ so that for any $\varphi_{0}\in L^{2}(\mathbb{R}^{N})$,
the solution $\varphi=\varphi(x,t)$ of (\ref{1.1}) satisfies
 \begin{equation}\label{3.33333}
\int_{\mathbb{R}^{N}}|\varphi(x,T)|^{2}\mathrm{d}x\leq e^{C_{8}
\left(T^{-1}+T+T\|a\|_{\infty}+\|a\|_{\infty}^{2/3}\right)}
\left(\int_{\mathbb{R}^{N}}|\varphi_{0}|^{2}\mathrm{d}x\right)^{\theta}
\left(\int_{\widetilde{\omega}}|\varphi(x,T)|^{2}\mathrm{d}x\right)^{1-\theta}.
 \end{equation}
\end{Lemma}
\begin{proof}
By Lemma~\ref{lemma-2.2} (where $r, R$ and $\delta$ are replaced by  $r, \sqrt{N}R$ and $1/2$, respectively), we obtain
\begin{eqnarray*}
 \int_{Q_{R}(\widetilde{x}_{i})}|\varphi(x,T)|^{2}\mathrm{d}x
 &\leq& \int_{B_{\sqrt{N}R}(\widetilde{x}_{i})}|\varphi(x,T)|^{2}\mathrm{d}x\\
 &\leq&\left[\mathcal{\widehat{K}}_{1}e^{[1+2C_{1}(1+R^{-2})](1+4T^{-1}+
 \|a\|^{2/3}_{\infty})+\mathcal{\widehat{K}}_{2}T^{-1}
 +2T\|a\|_{\infty}}\int_{T/2}^{T}\int_{Q_{4\sqrt{N}R}(\widetilde{x}_{i})}\varphi^{2}\mathrm{d}x\mathrm{d}t\right]^{\theta}\\
 &&\times\left[2\int_{B_{r}(\widetilde{x}_{i})}|\varphi(x,T)|^{2}\mathrm{d}x\right]^{1-\theta},
 \end{eqnarray*}
where $\mathcal{\widehat{K}}_{1}\triangleq\mathcal{\widehat{K}}_{1}(R)>0,
\mathcal{\widehat{K}}_{2}\triangleq\mathcal{\widehat{K}}_{2}(R)>0$ and
$\theta\triangleq\theta(r,R)\in (0,1).$ This, along with Young's  inequality,
implies that for each $\varepsilon>0,$
\begin{eqnarray*}
 \int_{Q_{R}(\widetilde{x}_{i})}|\varphi(x,T)|^{2}\mathrm{d}x
 &\leq&\varepsilon\theta\mathcal{\widehat{K}}_{1}e^{[1+2C_{1}(1+R^{-2})]
 (1+4T^{-1}+\|a\|^{2/3}_{\infty})+\mathcal{\widehat{K}}_{2}T^{-1}+2T\|a\|_{\infty}}
 \int_{T/2}^{T}\int_{Q_{4\sqrt{N}R}(\widetilde{x}_{i})}\varphi^{2}\mathrm{d}x\mathrm{d}t\\
 &&+2\varepsilon^{-\frac{\theta}{1-\theta}}(1-\theta)\int_{B_{r}(\widetilde{x}_{i})}
 |\varphi(x,T)|^{2}\mathrm{d}x.
 \end{eqnarray*}
Then
\begin{equation}\label{3.44444}
\begin{array}{lll}
 &&\displaystyle{}\int_{\mathbb{R}^{N}}|\varphi(x,T)|^{2}\mathrm{d}x
 =\sum_{i\geq1}\int_{Q_{R}(\widetilde{x}_{i})}|\varphi(x,T)|^{2}\mathrm{d}x\\
 \\
 &\leq&\displaystyle{}\varepsilon\theta\mathcal{\widehat{K}}_{1}
 e^{[1+2C_{1}(1+R^{-2})](1+4T^{-1}+\|a\|^{2/3}_{\infty})
 +\mathcal{\widehat{K}}_{2}T^{-1}+2T\|a\|_{\infty}}\sum_{i\geq 1}\int_{T/2}^{T}\int_{Q_{4\sqrt{N}R}(\widetilde{x}_{i})}\varphi^{2}\mathrm{d}x\mathrm{d}t\\
 \\
 &&\displaystyle{}+2\varepsilon^{-\frac{\theta}{1-\theta}}(1-\theta)\int_{\widetilde{\omega}}|\varphi(x,T)|^{2}\mathrm{d}x.
\end{array}
\end{equation}
Since
\begin{equation*}
\sum_{i\geq1}\int_{T/2}^{T}\int_{Q_{4\sqrt{N}R}(\widetilde{x}_{i})}\varphi^{2}
\mathrm{d}x\mathrm{d}t\leq \mathcal{\widehat{K}}_{3} \int_{T/2}^{T}
\int_{\mathbb{R}^{N}}\varphi^{2}\mathrm{d}x\mathrm{d}t,
 \end{equation*}
 where $\mathcal{\widehat{K}}_{3}>0$, it follows from (\ref{3.44444}) that
 \begin{eqnarray*}
 \int_{\mathbb{R}^{N}}|\varphi(x,T)|^{2}\mathrm{d}x
 &\leq&\varepsilon\theta\mathcal{\widehat{K}}_{1}\mathcal{\widehat{K}}_{3}
 e^{[1+2C_{1}(1+R^{-2})](1+4T^{-1}+\|a\|^{2/3}_{\infty})
 +\mathcal{\widehat{K}}_{2}T^{-1}+2T\|a\|_{\infty}}
 \int_{T/2}^{T}\int_{\mathbb{R}^{N}}\varphi^{2}\mathrm{d}x\mathrm{d}t\\
 &&+2\varepsilon^{-\frac{\theta}{1-\theta}}(1-\theta)
 \int_{\widetilde{\omega}}|\varphi(x,T)|^{2}\mathrm{d}x  \ \mathrm{for \ each} \ \varepsilon>0.
 \end{eqnarray*}
 This implies 
 \begin{equation}\label{3.444440}
\begin{array}{lll}
 &&\displaystyle{}\int_{\mathbb{R}^{N}}|\varphi(x,T)|^{2}\mathrm{d}x\\
 \\
 &\leq&\displaystyle{}\left[\mathcal{\widehat{K}}_{1}\mathcal{\widehat{K}}_{3}
 e^{[1+2C_{1}(1+R^{-2})](1+4T^{-1}+\|a\|^{2/3}_{\infty})
 +\mathcal{\widehat{K}}_{2}T^{-1}+2T\|a\|_{\infty}}
 \int_{T/2}^{T}\int_{\mathbb{R}^{N}}\varphi^{2}\mathrm{d}x\mathrm{d}t\right]^{\theta}\\
 \\
 &&\displaystyle{}\times\left[2\int_{\widetilde{\omega}}|\varphi(x,T)|^{2}\mathrm{d}x\right]^{1-\theta}.
\end{array}
\end{equation}

 Noting that
 $$
 \int_{\mathbb{R}^N} |\varphi(x,t)|^2\,\mathrm dx\leq e^{2\|a\|_{\infty}t}
 \int_{\mathbb{R}^N} |\varphi_0(x)|^2\,\mathrm dx\  \ \ \mathrm{for \  each} \  t\in [0,T],
 $$
 by (\ref{3.444440}), we deduce
 \begin{eqnarray*}
 \int_{\mathbb{R}^{N}}|\varphi(x,T)|^{2}\mathrm{d}x
 &\leq&\left[\mathcal{\widehat{K}}_{1}\mathcal{\widehat{K}}_{3}
 Te^{[1+2C_{1}(1+R^{-2})](1+4T^{-1}+\|a\|^{2/3}_{\infty})
 +\mathcal{\widehat{K}}_{2}T^{-1}+2T\|a\|_{\infty}}e^{2T\|a\|_{\infty}}
 \int_{\mathbb{R}^{N}}\varphi_{0}^{2}\mathrm{d}x\right]^{\theta}\\
 &&\times\left[2\int_{\widetilde{\omega}}|\varphi(x,T)|^{2}\mathrm{d}x\right]^{1-\theta}.
 \end{eqnarray*}
 Hence, (\ref{3.33333}) follows from the latter inequality immediately.\\

 In summary, we finish the proof of this lemma.
\end{proof}

Now, we are able to present the proof of Theorem~\ref{Thm1} by using the telescoping series method. For the convenience of the reader, we provide here the detailed computations  although it is more or less similar to that of \cite[Theorem 4]{Phung-Wang-Zhang}.\\

\noindent\textbf{Proof\ of\ Theorem~\ref{Thm1}}.
For any $0\leq t_{1}<t_{2}\leq T$, by a translation in the time variable and Lemma~\ref{2.3} (where $r, R, \widetilde{x}_{i}$ and $\widetilde{\omega}_{i}$ are replaced by
$r_{1}, r_{2}, x_{i}$ and $\omega_{i}$, respectively), we obtain from Young's inequality that
 \begin{equation}\label{2019-7-9}
 \|\varphi(t_{2})\|^{2}_{L^{2}(\mathbb{R}^{N})}\leq\varepsilon
 \|\varphi(t_{1})\|^{2}_{L^{2}(\mathbb{R}^{N})}+
 \frac{\mathcal{\widetilde{K}}_{1}}{\varepsilon^{\alpha}}e^{\frac{\mathcal{\widetilde{K}}_{2}}{t_{2}-t_{1}}}
 \|\varphi(t_{2})\|^{2}_{L^{2}(\omega)} \ \ \ \mathrm{for\ each}\ \varepsilon>0,
 \end{equation}
where $\mathcal{\widetilde{K}}_{1}\triangleq e^{\frac{C_8}{1-\theta}
\left(T+T\|a\|_{\infty}+\|a\|_{\infty}^{2/3}\right)},$
$\mathcal{\widetilde{K}}_{2}\triangleq C_{8}/(1-\theta)$
and $\alpha\triangleq\theta/(1-\theta)$.

Let $l$ be a density point of $E$. According to Proposition 2.1 in \cite{Phung-Wang},
for each $\kappa>1$, there exists $l_{1}\in (l,T)$, depending on $\kappa$ and $E$,
so that the sequence $\{l_{m}\}_{m\geq1}$, given by
$$
l_{m+1}=l+\frac{1}{\kappa^{m}}(l_{1}-l),
$$
satisfies 
 \begin{equation}\label{3.2525251}
l_{m}-l_{m+1}\leq 3|E\cap(l_{m+1},l_{m})|.
 \end{equation}

Next, let $0<l_{m+2}<l_{m+1}\leq t<l_{m}<l_{1}<T$. It follows from (\ref{2019-7-9}) that
\begin{equation}\label{3.2525252}
\|\varphi(t)\|^{2}_{L^{2}(\mathbb{R}^{N})}\leq \varepsilon\|\varphi(l_{m+2})\|^{2}_{L^{2}(\mathbb{R}^{N})}
+\frac{\mathcal{\widetilde{K}}_{1}}{\varepsilon^{\alpha}}
e^{\frac{\mathcal{\widetilde{K}}_{2}}{t-l_{m+2}}}\|\varphi(t)\|^{2}_{L^{2}(\omega)} \ \mathrm{for\ each}\ \varepsilon>0.
 \end{equation}
By a standard energy estimate,  we have
$$
\|\varphi(l_{m})\|_{L^{2}(\mathbb{R}^{N})}\leq e^{T\|a\|_{\infty}}\|\varphi(t)\|_{L^{2}(\mathbb{R}^{N})}.
$$
This, along with (\ref{3.2525252}), implies
$$
\|\varphi(l_{m})\|^{2}_{L^{2}(\mathbb{R}^{N})}\leq e^{2T\|a\|_{\infty}}\left(\varepsilon\|\varphi(l_{m+2})
\|^{2}_{L^{2}(\mathbb{R}^{N})}+
\frac{\mathcal{\widetilde{K}}_{1}}{\varepsilon^{\alpha}}
e^{\frac{\mathcal{\widetilde{K}}_{2}}{t-l_{m+2}}}\|\varphi(t)\|^{2}_{L^{2}(\omega)}\right)
\ \mathrm{for\ each}\ \varepsilon>0,
$$
which indicates that
$$\|\varphi(l_{m})\|^{2}_{L^{2}(\mathbb{R}^{N})}
\leq \varepsilon\|\varphi(l_{m+2})\|^{2}_{L^{2}(\mathbb{R}^{N})}
+\frac{\mathcal{\widetilde{K}}_{3}}{\varepsilon^{\alpha}}
e^{\frac{\mathcal{\widetilde{K}}_{2}}{t-l_{m+2}}}\|\varphi(t)\|^{2}_{L^{2}(\omega)}
 \ \mathrm{for\ each}\ \varepsilon>0,
 $$
where $\mathcal{\widetilde{K}}_{3}=(e^{2T\|a\|_{\infty}})^{1+\alpha}\mathcal{\widetilde{K}}_{1}$.
Integrating the latter inequality over $ E\cap(l_{m+1},l_{m})$ gives
\begin{equation}\label{3.2525253}
\begin{array}{lll}
 \displaystyle{}|E\cap(l_{m+1},l_{m})|\|\varphi(l_{m})\|^{2}_{L^{2}(\mathbb{R}^{N})}
 &\leq&\displaystyle{}\varepsilon |E\cap(l_{m+1},l_{m})|\|\varphi(l_{m+2})\|^{2}_{L^{2}(\mathbb{R}^{N})}\\
 &&\displaystyle{}+\frac{\mathcal{\widetilde{K}}_{3}}
 {\varepsilon^{\alpha}}e^{\frac{\mathcal{\widetilde{K}}_{2}}{l_{m+1}-l_{m+2}}}
 \int_{l_{m+1}}^{l_{m}}\chi_{E}\|\varphi(t)\|^{2}_{L^{2}(\omega)}\mathrm{d}t
 \;\ \mathrm{for\ each}\ \varepsilon>0.
\end{array}
\end{equation}
Here and in the sequel, $\chi_{E}$ denotes the characteristic function of $E$.

 Since $l_{m}-l_{m+1}=(\kappa-1)(l_{1}-l)/\kappa^{m},$ by (\ref{3.2525253}) and (\ref{3.2525251}), we obtain 
\begin{eqnarray*}
\|\varphi(l_{m})\|^{2}_{L^{2}(\mathbb{R}^{N})}&\leq& \varepsilon \|\varphi(l_{m+2})\|^{2}_{L^{2}(\mathbb{R}^{N})}+\frac{1}{|E\cap(l_{m+1},l_{m})|}
\frac{\mathcal{\widetilde{K}}_{3}}{\varepsilon^{\alpha}}
e^{\frac{\mathcal{\widetilde{K}}_{2}}{l_{m+1}-l_{m+2}}}
\int_{l_{m+1}}^{l_{m}}\chi_{E}\|\varphi(t)\|^{2}_{L^{2}(\omega)}\mathrm{d}t\\
&\leq&\frac{3\kappa^{m}}{(l_{1}-l)(\kappa-1)}
\frac{\mathcal{\widetilde{K}}_{3}}{\varepsilon^{\alpha}}
e^{\mathcal{\widetilde{K}}_{2}\left(\frac{1}{l_{1}-l}\frac{\kappa^{m+1}}{\kappa-1}\right)}
\int_{l_{m+1}}^{l_{m}}\chi_{E}\|\varphi(t)\|^{2}_{L^{2}(\omega)}\mathrm{d}t+
\varepsilon \|\varphi(l_{m+2})\|^{2}_{L^{2}(\mathbb{R}^{N})}
 \end{eqnarray*}
 for each $\varepsilon>0$.
This yields 
\begin{equation}\label{3.2525254}
\begin{array}{lll}
 \displaystyle{}\|\varphi(l_{m})\|^{2}_{L^{2}(\mathbb{R}^{N})}&\leq& \displaystyle{}
 \frac{1}{\varepsilon^{\alpha}}\frac{3}{\kappa}
 \frac{\mathcal{\widetilde{K}}_{3}}{\mathcal{\widetilde{K}}_{2}}
 e^{2\mathcal{\widetilde{K}}_{2}\left(\frac{1}{l_{1}-l}
 \frac{\kappa^{m+1}}{\kappa-1}\right)}\int_{l_{m+1}}^{l_{m}}\chi_{E}
 \|\varphi(t)\|^{2}_{L^{2}(\omega)}\mathrm{d}t\displaystyle{}+
 \varepsilon \|\varphi(l_{m+2})\|^{2}_{L^{2}(\mathbb{R}^{N})}
\end{array}
\end{equation}
for each $\varepsilon>0$.
Denote by $d\triangleq \frac{2\mathcal{\widetilde{K}}_{2}}{\kappa(l_{1}-l)(\kappa-1)}$.
It follows from (\ref{3.2525254}) that
\begin{eqnarray*}
\varepsilon^{\alpha}e^{-d\kappa^{m+2}}\|\varphi(l_{m})\|^{2}_{L^{2}(\mathbb{R}^{N})}
-\varepsilon^{1+\alpha}e^{-d\kappa^{m+2}}\|\varphi(l_{m+2})\|^{2}_{L^{2}(\mathbb{R}^{N})}
\leq\frac{3}{\kappa}\frac{\mathcal{\widetilde{K}}_{3}}
{\mathcal{\widetilde{K}}_{2}}\int_{l_{m+1}}^{l_{m}}\chi_{E}\|\varphi(t)\|^{2}_{L^{2}(\omega)}\mathrm{d}t
\end{eqnarray*}
for each $\varepsilon>0$.

Choosing $\varepsilon=e^{-d\kappa^{m+2}}$ in the above inequality gives 
\begin{equation}\label{3.25252555}
\begin{array}{lll}
 &&\displaystyle{}e^{-(1+\alpha)d\kappa^{m+2}}\|\varphi(l_{m})\|^{2}_{L^{2}(\mathbb{R}^{N})}
 -e^{-(2+\alpha)d\kappa^{m+2}}\|\varphi(l_{m+2})\|^{2}_{L^{2}(\mathbb{R}^{N})}\\
 &\leq&\displaystyle{}\frac{3}{\kappa}\frac{\mathcal{\widetilde{K}}_{3}}
 {\mathcal{\widetilde{K}}_{2}}\int_{l_{m+1}}^{l_{m}}\chi_{E}\|\varphi(t)\|^{2}_{L^{2}(\omega)}\mathrm{d}t.
\end{array}
\end{equation}
Take $\kappa=\sqrt{\frac{\alpha+2}{\alpha+1}}$ in (\ref{3.25252555}). Then we have 
\begin{eqnarray*}
e^{-(2+\alpha)d\kappa^{m}}\|\varphi(l_{m})\|^{2}_{L^{2}(\mathbb{R}^{N})}
-e^{-(2+\alpha)d\kappa^{m+2}}\|\varphi(l_{m+2})\|^{2}_{L^{2}(\mathbb{R}^{N})}
\leq \frac{3}{\kappa}\frac{\mathcal{\widetilde{K}}_{3}}{\mathcal{\widetilde{K}}_{2}}
\int_{l_{m+1}}^{l_{m}}\chi_{E}\|\varphi(t)\|^{2}_{L^{2}(\omega)}\mathrm{d}t.
\end{eqnarray*}
Changing $m$ to $2m'$ and summing the above inequality from $m'=1$ to infinity give the desired result. Indeed,
\begin{eqnarray*}
&&e^{-2T\|a\|_{\infty}}e^{-(2+\alpha)d\kappa^{2}}\|\varphi(T)\|^{2}_{L^{2}(\mathbb{R}^{N})}\\
&\leq& e^{-(2+\alpha)d\kappa^{2}}\|\varphi(l_{2})\|^{2}_{L^{2}(\mathbb{R}^{N})}\\
&\leq&\sum_{m'=1}^{+\infty}\left(e^{-(2+\alpha)d\kappa^{2m'}}\|\varphi(l_{2m'})\|_{L^{2}(\mathbb{R}^{N})}
-e^{-(2+\alpha)d\kappa^{2m'+2}}\|\varphi(l_{2m'+2})\|^{2}_{L^{2}(\mathbb{R}^{N})}\right)\\
&\leq& \frac{3}{\kappa}\frac{\mathcal{\widetilde{K}}_{3}}{\mathcal{\widetilde{K}}_{2}}\sum_{m'=1}^{+\infty}
\int_{l_{2m'+1}}^{l_{2m'}}\chi_{E}\|\varphi(t)\|^{2}_{L^{2}(\omega)}\mathrm{d}t\\
&\leq& \frac{3}{\kappa}\frac{\mathcal{\widetilde{K}}_{3}}
{\mathcal{\widetilde{K}}_{2}}\int_{0}^{T}\chi_{E}\|\varphi(t)\|^{2}_{L^{2}(\omega)}\mathrm{d}t.
 \end{eqnarray*}

In summary, we finish the proof of Theorem~\ref{Thm1}.
\qed

\vskip 3mm

We show an application of Theorem~\ref{Thm1} on the null controllability from measurable sets in the time variable. The latter plays an important role in deriving the bang-bang property for the time optimal control problem (see, e.g., \cite{AEWZ, Phung-Wang}). Under the same assumptions of Theorem~\ref{Thm1},
we consider the following controlled equation:
\begin{equation}\label{nullcontrol22}
\left\{ \begin{array}{lll}
\partial_{t}y-\Delta y+b(x,t)y=\chi_{\omega}\chi_{E}u\ \ \ \ \ \ \ \ \mathrm{in}\  \mathbb{R}^{N}\times (0,T),\\
y(0)=y_{0} \ \ \ \ \ \ \ \ \ \ \ \ \ \ \ \ \ \ \ \ \ \ \ \ \ \ \ \ \ \ \ \ \ \ \mathrm{in}\ \mathbb{R}^{N},\\
\end{array}\right.\end{equation}
where $y_{0}\in L^2(\mathbb{R}^{N})$ is an initial state, $b\in L^{\infty}(\mathbb{R}^{N}\times(0,T))$, and 
$u\in L^{2}(\mathbb{R}^{N}\times(0,T))$ is a control function.
Write $y(\cdot; y_{0}, u)$ for the solution to (\ref{nullcontrol22}).
By a standard duality method (see, for instance, \cite{WangGengsheng}) and Theorem~\ref{Thm1},
we can easily obtain the following null controllability result. (Its proof will be omitted here.)
\begin{corollary}\label{Thm22} Under the assumption of  Theorem~\ref{Thm1},
for each $y_{0}\in L^2(\mathbb{R}^{N})$,
there is a control $u\in L^{2}(\mathbb{R}^{N}\times(0,T))$, with
\begin{equation*}
\|u\|_{L^2(0,T;L^2(\mathbb{R}^{N}))}\leq
e^{\widetilde{C}} e^{C\left(T+T\|b\|_{\infty}+\|b\|_{\infty}^{2/3}\right)}\|y_{0}\|_{L^{2}(\mathbb{R}^{N})},
  \end{equation*}
where the constants $C$ and $\widetilde{C}$ are given by  Theorem~\ref{Thm1}, so that $y(T; y_{0},u)=0.$
\end{corollary}

\begin{remark}
It is interesting to ask the following question: whether the null controllability for semilinear heat equations in
$\mathbb{R}^N$ with the control acted on the equidistributed set $\omega$ holds?
It is well-known that the null and approximate controllability were proved for semilinear
heat equations in a bounded domain $\Omega$ (see, e.g., \cite{Fabre,fz}).
Roughly speaking, their proofs consist of two parts: (i) null and approximate controllability of
the linearized system; (ii) fixed-point theory. When $\Omega$ is a general unbounded domain, however,
the above approach cannot be directly  applied  because of the lack of compactness of Sobolev's embedding, which is
one of the main ingredients used in (ii). Instead, the authors of \cite{TeresaDe} studied
the approximate controllability of a semilinear heat equation in an unbounded domain
$\Omega$ by an approximation method. More precisely, they first considered the control problem in bounded domains of the form $\Omega_r\triangleq\Omega\cap B_r$, where $B_r$ denotes the ball centered
at the origin and of radius $r$. They then showed that the controls proposed in
\cite{Fabre} restricted to $\Omega_r$ converge in certain sense to a desired approximate control in
the whole domain. The approximate controllability of a semilinear heat
equation in $\mathbb{R}^N$ was also considered in \cite{ldt}, where the author introduced  weighted Sobolev spaces
and adapted the technique introduced by \cite{Fabre}.  Inspired by the ideas in the works  \cite{TeresaDe} and \cite{ldt}, we tried to use Corollary \ref{Thm22}  to prove  the null controllability for a semilinear heat equation in
$\mathbb{R}^N$ with the control acted on $\omega$.  
By our understanding, one may need to improve our main result in the following two ways:
(i)  the dependence of observability constant (in the observability inequality) on $r$;
(ii) a suitable weighted observability inequality.
The authors  hope to  explore them  by introducing some new ideas in the future work. 
\end{remark}

\medskip

We end this section with an interesting observation.
According to Lemma~\ref{lemma-1.1}, it is clear that for any $r< R$, the solution  $\varphi$ of (\ref{1.1}) satisfies 
\begin{equation}\label{1.2999*}
\begin{array}{lll}
\displaystyle{}\int_{B_{r}(x_{0})}\varphi^{2}(x,T)\mathrm{d}x\leq
\displaystyle{C_{1}}\left[(R-r)^{-2}+4T^{-1}+\|a\|_{\infty}\right]
\int_{0}^{T}\int_{B_{R}(x_{0})}\varphi^{2}\mathrm{d}x\mathrm{d}t.\\
\end{array}
\end{equation}
Under the same assumptions of Lemma~\ref{lemma-1.1}, we consider the following controlled equation:
\begin{equation}\label{nullcontrol23}
\left\{ \begin{array}{lll}
\partial_{t}z-\Delta z+b(x,t)z=\chi_{B_{R}(x_{0})}v\ \ \ \ \ \ \ \ \mathrm{in}\  \mathbb{R}^{N}\times (0,T),\\
z(0)=z_{0} \ \ \ \ \ \ \ \ \ \ \ \ \ \ \ \ \ \ \ \ \ \ \ \ \ \ \ \ \ \ \ \ \ \ \ \ \mathrm{in}\ \mathbb{R}^{N},\\
\end{array}\right.\end{equation}
where $z_{0}\in L^2(\mathbb{R}^{N})$ is an initial state,
$v\in L^{2}(\mathbb{R}^{N}\times(0,T))$ is a control and $b\in L^{\infty}(\mathbb{R}^{N}\times(0,T))$.
Write $z(\cdot; z_{0}, v)$ for the solution to (\ref{nullcontrol23}).
By a standard duality method (see also \cite{WangGengsheng}) and (\ref{1.2999*}),
we can obtain the following null controllability result. (Its proof will be omitted.)

\begin{corollary}\label{Thm3} Under the assumptions of Lemma~\ref{lemma-1.1},
for each $z_{0}\in L^{2}(\mathbb{R}^{N})$, with $\mathrm{supp} \ z_{0}\subset B_{r}(x_{0}),$
there is a control $v\in L^{2}(\mathbb{R}^{N}\times(0,T))$, with
\begin{equation*}
\|v\|_{L^2(0,T;L^2(\mathbb{R}^{N}))}\leq \displaystyle{C_{1}}\left[(R-r)^{-2}+
4T^{-1}+\|b\|_{\infty}\right]\|z_{0}\|_{L^{2}(\mathbb{R}^{N})},
  \end{equation*}
where the constant  $C_{1}$ is given by  Lemma~\ref{lemma-1.1}, so that  $z(T; z_{0}, v)=0.$
\end{corollary}
\begin{remark}
It should be pointed out that the same null controllability stated in Corollary~\ref{Thm3}
was already established in \cite{WangZhangZhang} for the case that $b\equiv 0$.
\end{remark}

\medskip

\section{Appendix}
\noindent\textbf{Proof\ of\ Lemma~\ref{lemma-1.1}}.
For simplicity we write  $B_{r}\triangleq B_{r}(x_{0})$ and  $B_{R}\triangleq B_{R}(x_{0}).$
Let $\eta\in C_{0}^{\infty}(B_{R})$ verify
\begin{equation}\label{2.21111}
0\leq\eta(\cdot)\leq 1 \ \mathrm{in} \ B_{R}, \ \eta(\cdot)=1 \ \mathrm{in} \ B_{r} \
\mathrm{and} \  |\nabla \eta(\cdot)|\leq C(R-r)^{-1}.
\end{equation}
Here and throughout the proof of Lemma~\ref{lemma-1.1},
$C$ denotes a generic positive constant. Let $\xi\in C^{\infty}(\mathbb{R})$ satisfy
\begin{equation}\label{2.31111}
0\leq\xi(\cdot)\leq1, \ |\xi'(\cdot)|\leq C(\tau_{2}-\tau_{1})^{-1} \ \ \
\mathrm{in} \ \ \mathbb{R},
  \end{equation}
\begin{equation}\label{2.41111}
\xi(\cdot)=0 \ \mathrm{in} \ (-\infty, T-\tau_{2}] \ \mathrm{and} \ \xi(\cdot)=1 \
\mathrm{in} \ [T-\tau_{1}, +\infty).
  \end{equation}
Multiplying the first equation of (\ref{1.1}) by $\eta^{2}\xi^{2}\varphi$ and
integrating it over  $B_{R}\times(T-\tau_{2},t)$ for  $t\in [T-\tau_{1},T]$, we obtain 
\begin{equation}\label{2.51111}
\begin{array}{lll}
&&\displaystyle{\frac{1}{2}}\int_{B_{R}}\eta^{2}\xi^{2}(t)\varphi^{2}(x,t) \mathrm{d}x+\int_{T-\tau_{2}}^{t}\int_{B_{R}} \eta^{2}\xi^{2}|\nabla\varphi|^{2} \mathrm{d}x\mathrm{d}s\\
\\
&=&\displaystyle{-2}\int_{T-\tau_{2}}^{t}\int_{B_{R}}\xi^{2}\eta\varphi\nabla \eta\cdot\nabla\varphi\mathrm{d}x\mathrm{d}s
+\int_{T-\tau_{2}}^{t}\int_{B_{R}} \eta^{2}\xi\xi'\varphi^{2} \mathrm{d}x\mathrm{d}s\\
\\
&&\displaystyle{-}\int_{T-\tau_{2}}^{t}\int_{B_{R}} a \eta^{2}\xi^{2}\varphi^{2} \mathrm{d}x\mathrm{d}s.
\end{array}
\end{equation}
Applying Young's inequality to the first term on the 
right hand of (\ref{2.51111}), we have 
\begin{eqnarray*}
&&\int_{B_{R}}\eta^{2}\xi^{2}(t)\varphi^{2}(x,t) \mathrm{d}x
+\int_{T-\tau_{2}}^{t}\int_{B_{R}} \eta^{2}\xi^{2}|\nabla\varphi|^{2} \mathrm{d}x\mathrm{d}s\\
&\leq&4\int_{T-\tau_{2}}^{t}\int_{B_{R}} |\nabla\eta|^{2}\xi^{2}\varphi^{2} \mathrm{d}x\mathrm{d}s
+2\int_{T-\tau_{2}}^{t}\int_{B_{R}} \eta^{2}\xi\xi'\varphi^{2} \mathrm{d}x\mathrm{d}s\\
&&-2\int_{T-\tau_{2}}^{t}\int_{B_{R}} a \eta^{2}\xi^{2}\varphi^{2} \mathrm{d}x\mathrm{d}s.
\end{eqnarray*}
This, along with  (\ref{2.21111})-(\ref{2.41111}), implies that
\begin{eqnarray*}
&&\int_{B_{r}}\varphi^{2}(x,t) \mathrm{d}x+\int_{T-\tau_{1}}^{t}\int_{B_{r}} |\nabla\varphi|^{2} \mathrm{d}x\mathrm{d}s\\
\\
&\leq&C\left[(R-r)^{-2}+(\tau_{2}-\tau_{1})^{-1}+\|a\|_{\infty}\right]
\int_{T-\tau_{2}}^{T}\int_{B_{R}}\varphi^{2} \mathrm{d}x\mathrm{d}s\ \  \mathrm{for\ each} \ t\in [T-\tau_{1},T].
\end{eqnarray*}
Hence, (\ref{1.2}) follows from the last inequality immediately.
\qed

\medskip
\noindent\textbf{Proof\ of\ Lemma~\ref{lemma-1.2}}.
For each $r'>0,$ we write $B_{r'}\triangleq B_{r'}(x_{0})$.  Let $\eta\in C_{0}^{\infty}(B_{4R/3})$ satisfy
\begin{equation}\label{2.61111}
0\leq \eta(\cdot)\leq 1,\ |\nabla\eta(\cdot)|\leq CR^{-1}, \ |\Delta \eta(\cdot)|\leq CR^{-2}\ \mathrm{ in} \ B_{4R/3}
\end{equation}
and
\begin{equation}\label{2.61112}
\eta(\cdot)=1 \ \mathrm{ in} \ B_{R}.
\end{equation}
Here and throughout the 
proof of Lemma~\ref{lemma-1.2}, $C$ denotes a generic positive constant.
Let $\xi\in C^{\infty}(\mathbb{R})$ verify
\begin{equation}\label{2.61113}
0\leq \xi(\cdot)\leq1,\ |\xi'(\cdot)|\leq C\tau^{-1} \ \mathrm{ in} \ \mathbb{R},
\end{equation}
\begin{equation}\label{2.61114}
\xi(\cdot)=0 \ \mathrm{in} \ (-\infty, T-4\tau/3] \ \mathrm{and } \ \xi(\cdot)=1  \ \mathrm{in} \ [T-\tau, +\infty).
\end{equation}
Denote by $z\triangleq\eta\xi\varphi$. It is easy to check that
\begin{equation}\label{2.71111}
\left\{
\begin{array}{lll}
\partial_{t}z-\Delta z=(\eta\xi'-\xi\Delta\eta-a\eta\xi)\varphi-2\xi\nabla\eta\cdot\nabla\varphi&\mathrm{in}\ B_{4R/3}\times (0,T),\\
z=0&\mathrm{on}\ \partial B_{4R/3}\times (0,T),\\
z(T-4\tau/3)=0&\mathrm{in}\ B_{4R/3}.\\
\end{array}
\right.
\end{equation}

On one hand, for each $t\in[T-\tau, T],$ we have 
\begin{equation*}
-2\int^{t}_{T-4\tau/3}\int_{B_{4R/3}}\Delta z \partial_{s}z\mathrm{d}x\mathrm{d}s=\int_{B_{4R/3}}|\nabla z(x,t)|^{2}\mathrm{d}x-\int_{B_{4R/3}}|\nabla z(x,T-4\tau/3)|^{2}\mathrm{d}x,
\end{equation*}
which indicates 
$$
\int_{B_{4R/3}}|\nabla z(x,t)|^{2}\mathrm{d}x\leq
\int^{t}_{T-4\tau/3}\int_{B_{4R/3}}(\Delta z- \partial_{s}z)^2\mathrm{d}x\mathrm{d}s
\;\;\mbox{for each}\;\;t\in[T-\tau, T].
$$
This, along with (\ref{2.61112}) and the second relation of (\ref{2.61114}), implies that
\begin{equation}\label{1.4}
\displaystyle{\max_{t\in[T-\tau,T]}}\int_{B_{R}}|\nabla \varphi(x,t)|^{2}\mathrm{d}x\leq
\int^{T}_{T-4\tau/3}\int_{B_{4R/3}}(\Delta z- \partial_{s}z)^2\mathrm{d}x\mathrm{d}s.
\end{equation}

On the other hand, 
\begin{equation}\label{1.5}
\begin{array}{lll}
&&\displaystyle{}\int_{T-4\tau/3}^{T}\int_{B_{4R/3}}
\left[(\eta\xi'-\xi\Delta\eta-a\eta\xi)\varphi-2\xi\nabla\eta\cdot\nabla\varphi\right]^{2}\mathrm{d}x\mathrm{d}t\\
\\
&\leq&\displaystyle{8}\int_{T-4\tau/3}^{T}\int_{B_{4R/3}}\left[(\eta^2|\xi'|^2+\xi^2|\Delta \eta|^2+a^{2}\eta^2\xi^2)\varphi^{2}
+\xi^2|\nabla \eta|^2|\nabla \varphi|^{2}\right]\mathrm{d}x\mathrm{d}t.
\end{array}
\end{equation}
By (\ref{1.5}), (\ref{2.61111}) and (\ref{2.61113}), we get 
\begin{equation}\label{1.6}
\begin{array}{lll}
&&\displaystyle{}\int_{T-4\tau/3}^{T}\int_{B_{4R/3}}
\left[(\eta\xi'-\xi\Delta\eta-a\eta\xi)\varphi-2\xi\nabla\eta\cdot\nabla\varphi\right]^{2}\mathrm{d}x\mathrm{d}t\\
\\
&\leq&\displaystyle{C}\left(\tau^{-2}+R^{-4}+\|a\|^{2}_{\infty}\right)
\int_{T-4\tau/3}^{T}\int_{B_{4R/3}}\varphi^2\mathrm{d}x\mathrm{d}t
+CR^{-2}\int_{T-4\tau/3}^{T}\int_{B_{4R/3}}|\nabla\varphi|^2\mathrm{d}x\mathrm{d}t.
\end{array}
\end{equation}

According to (\ref{1.2}) (where $r, R, \tau_{1}$ and $\tau_{2}$ are replaced by
$4R/3, 2R, 4\tau/3$ and $2\tau$, respectively), it is clear that
$$\int_{T-4\tau/3}^{T}\int_{B_{4R/3}}| \nabla\varphi|^{2}\mathrm{d}x\mathrm{d}t\leq C\left(\tau^{-1}+R^{-2}+\|a\|_{\infty}\right)\int_{T-2\tau}^{T}\int_{B_{2R}}\varphi^{2}\mathrm{d}x\mathrm{d}t.$$
This, along with (\ref{1.6}), implies that
\begin{eqnarray*}
&&\displaystyle{}\int_{T-4\tau/3}^{T}\int_{B_{4R/3}}
\left[(\eta\xi'-\xi\Delta\eta-a\eta\xi)\varphi-2\xi\nabla\eta\cdot\nabla\varphi\right]^{2}\mathrm{d}x\mathrm{d}t\\
\\
&\leq&C\left(\tau^{-2}+R^{-4}+\|a\|^{2}_{\infty}\right)\int_{T-2\tau}^{T}\int_{B_{2R}}\varphi^{2}\mathrm{d}x\mathrm{d}t
+CR^{-2}\left(\tau^{-1}+R^{-2}+\|a\|_{\infty}\right)\int_{T-2\tau}^{T}\int_{B_{2R}}\varphi^{2}\mathrm{d}x\mathrm{d}t\\
\\
&\leq&C\left(\tau^{-2}+R^{-4}+\|a\|^{2}_{\infty}\right)\int_{T-2\tau}^{T}\int_{B_{2R}}\varphi^{2}\mathrm{d}x\mathrm{d}t.
\end{eqnarray*}
Hence, (\ref{1.3}) follows from the last inequality, (\ref{1.4}), and the first equation of (\ref{2.71111}).
\qed

\bigskip

\noindent\textbf{Acknowledgments}.
This work was partially supported by the National Natural Science Foundation of China under grants 11771344 and 11971363.  The third author would like to thank Prof. Luis Escauriaza for fruitful discussions related to this work when he was a postdoc in UPV/EHU.

\end{document}